\documentclass[11pt,a4paper,headinclude,footinclude,fleqn,reqno]{amsart}                 
\usepackage[T1]{fontenc}                   
\usepackage[utf8]{inputenc}                 
\usepackage[english]{babel}       
\usepackage{graphicx}                      % 
\usepackage[font=small]{quoting}            %  
\usepackage{caption}  
\usepackage{amsmath}
\usepackage{amsthm}
\usepackage{amssymb}            
\usepackage[top=.8in,bottom=1.2in,left=1.5in,right=1.5in]{geometry}
\usepackage{verbatim}
\usepackage{color}
\usepackage{picture}
\usepackage{graphicx}
\usepackage{graphics}
\usepackage{comment}
\usepackage{hyperref}
\hypersetup{colorlinks,linkcolor={blue},citecolor={blue},urlcolor={red}}  
\newtheorem{theorem}{\bf Theorem}[section]
\newtheorem{definition}[theorem]{\bf Definition}
\newtheorem{lemma}[theorem]{\bf Lemma}
\newtheorem{prop}[theorem]{\bf Proposition}

\def\no{\noindent}

\def\Om{\Omega}
\def\pa{\partial}

\def\ov{\overline}

\makeatletter

\newcommand{\Rmnum}[1]{\expandafter\@slowromancap\romannumeral #1@}
\makeatother \numberwithin{equation}{section}

\begin{document}

\begin{center}{\LARGE \bf Classification of conformal minimal immersions  from $S^2$ to $G(2,N;\mathbb{C})$ with parallel second fundamental form}
\end{center}

\begin{center}

Xiaoxiang Jiao
\footnote{
X.X. Jiao 

School of Mathematical Sciences, University of Chinese Academy of
Sciences, Beijing 100049, P. R. China

e-mail: xxjiao@ucas.ac.cn}
and
Mingyan Li
\footnote{
M.Y. Li (Corresponding author)

School of Mathematical Sciences, Ocean University of China, Qingdao 266100, P. R. China

e-mail: limingyan@ouc.edu.cn
}  
\end{center}

\bigskip

\no
{\bf Abstract.} In this paper, we determine all conformal minimal immersions of 2-spheres in
complex Grassmann manifold $G(2,N; \mathbb{C})$ with parallel second fundamental form.

\no

\no

\no
Project supported by the NSFC (Grant Nos. 11871450, 11901534).

\bigskip

\bigskip

\section{Introduction}\label{sec1}

The geometry of an $m$-dimensional smooth submanifold in an $n$-dimensional Riemann space is determined by two symmetric fundamental tensor fields (equivalently, quadratic differential forms): the first fundamental form, that is, the metric tensor, and the second fundamental form, the latter with values in normal vector bundle of the submanifold. It is well known that the first one is parallel by definition, but the second one does not need to be parallel. Therefore an interesting class of submanifolds 
 with parallel second fundamental tensor field can be singed out, and its classification is an enduring and important topic.

It is a long history of studying minimal submanifolds with parallel second fundamental form in various Riemannian spaces. The first result on parallel submanifolds was given by V.F. Kagan \cite{Kagan} in 1948 who showed that the class of parallel surfaces in 3-dimensional Euclidean space consists of open parts of planes, round spheres and circular cylinders $S^1\times \mathbb{R}^1$. Since then, there have emerged many works on parallel submanifolds in not only  Euclidean space, but also various Riemannian spaces (see \cite{Ferus, Ferus1, Naitoh, Takeuchi} and the references therein).  In an elegant paper \cite{Nakagawa},  H. Nakagawa and R. Takagi  studied some properties about
K\"{a}hler imbeddings of compact Hermitian symmetric spaces in
 complex projective space $\mathbb{C}P^n$ and gave a classification of K\"{a}hler submanifolds
in  $\mathbb{C}P^n$ with parallel second
fundamental form. In 1984 A. Ros \cite{Ros}
decided all compact Einstein K\"{a}hler submanifolds in $\mathbb{C}P^n$ with parallel second fundamental
form. Generally, studying classifications of conformal minimal
two-spheres  immersed in various Riemannian symmetric spaces
with parallel second fundamental form is very difficult. Recently, we discussed the geometry of
conformal minimal immersions from $S^2$ to the hyperquadric  $Q_n$  and gave a complete
classification theorem of them under the assumption that they have parallel second fundamental form (cf. \cite{JGA}). L. He and the first author also classified all conformal minimal two-spheres immersed in the quaternionic projective space $HP^n$ with parallel second fundamental form (cf. \cite{HPn1}).

Let $G(2,N;\mathbb{C})$ be the complex Grassmann manifold consisting of all complex 2-dimensional subspaces in the complex vector space $\mathbb{C}^N$. Regarding $\mathbb{C}P^{N-2}, \ Q_{N-2}$ and $HP^n \ (N=2n+2)$ are maximal totally geodesic submanifolds in  $G(2,N;\mathbb{C})$, it was natural to try to extend these results described above to
study the rigidity of harmonic maps from $S^2$ to $G(2,N;\mathbb{C})$. Let $\phi: S^2\rightarrow G(2,N;\mathbb{C})$ be a linearly full conformal minimal immersion with Gauss curvature $K$ and  second fundamental form $B$. Suppose $B$ is parallel, in this paper, we  firstly  investigate  geometry of $\phi$  by the theory
of harmonic maps and generalize  our  characterization of the harmonic sequence generated
by $\phi$. Then we mainly study the rigidity of $\phi$ and give its
 classification  according to  the following four cases:
\\ (I) $\phi$ is a holomorphic curve in $G(2,N;\mathbb{C})$;
\\ (II) rank $\partial{'}\underline{\phi} = $ rank $\partial{''}\underline{\phi} = 1$;
\\ (III) rank $\partial{''}\underline{\phi} = 2$ and rank $\partial{'}\underline{\phi} = 1$;
\\ (IV) rank $\partial{'}\underline{\phi} = $ rank $\partial{''}\underline{\phi} = 2$.

~\\
Our classification theorems of $\phi$ for cases (I)-(IV) are as follows respectively:

\begin{theorem}
Let $\phi: S^2\rightarrow G(2,N;\mathbb{C})$ be a  linearly full holomorphic curve, and let $K$ and $B$ be its Gauss
 curvature and second fundamental form respectively. If $B$ is
 parallel, then  $\phi$ belongs to one of the following
minimal immersions.
\\(1) up to $U(3)$ equivalence, $\phi$ is $\underline{V}^{(2)}_0\oplus \underline{V}^{(2)}_1: S^2 \rightarrow G(2,3; \mathbb{C})$ with $K=2$ and $\|B\|^2=4$;
\\(2) up to $U(3)$ equivalence, $\phi$ is $\underline{V}^{(1)}_0\oplus \underline{c}_0: S^2 \rightarrow G(2,3; \mathbb{C})$ with $K=4$ and $\|B\|^2=0$, where $\ c_0 = (0,0,1)^T$;
\\(3) up to $U(4)$ equivalence, $\phi$ is $\underline{V}^{(2)}_0\oplus \underline{c}_0: S^2 \rightarrow G(2,4; \mathbb{C})$ with $K=2$ and $\|B\|^2=4$, where $\ c_0 = (0,0,0,1)^T$;
\\(4) up to $U(4)$ equivalence, $\phi$ is $\widehat{\underline{V}}^{(1)}_0\oplus \underline{V}^{(1)}_0: S^2 \rightarrow G(2,4; \mathbb{C})$  with $K=2$ and $\|B\|^2=0$;
\\(5) up to $U(6)$ equivalence, $\phi$ is $\widehat{\underline{V}}^{(2)}_0\oplus \underline{V}^{(2)}_0: S^2 \rightarrow G(2,6; \mathbb{C})$  with $K=1$ and $\|B\|^2=2$.
\end{theorem}

\begin{theorem}
Let $\phi: S^2\rightarrow G(2,N;\mathbb{C})$ be a  linearly full conformal minimal immersion with rank $\partial{'}\underline{\phi} = $  rank $\partial{''}\underline{\phi} = 1$,
 and let $K$ and $B$ be its Gauss
 curvature and second fundamental form respectively. If $B$ is
 parallel, then  $\phi$ belongs to one of the following
minimal immersions.
\\(1) up to $U(4)$ equivalence, $\phi$ is $\underline{V}^{(3)}_1\oplus \underline{V}^{(3)}_2: S^2 \rightarrow G(2,4;\mathbb{C})$ with $K=\frac{2}{3}$ and $\|B\|^2=\frac{8}{3}$;
\\(2) up to $U(3)$ equivalence, $\phi$ is $\underline{V}^{(2)}_0\oplus \underline{V}^{(2)}_2: S^2 \rightarrow G(2,3; \mathbb{C})$ with $K=1$ and $\|B\|^2=0$;
\\(3) up to $U(4)$ equivalence, $\phi$ is $\widehat{\underline{V}}^{(1)}_0\oplus \underline{V}^{(1)}_1: S^2 \rightarrow G(2,4; \mathbb{C})$ with $K=2$ and $\|B\|^2=0$;
\\(4) up to $U(4)$ equivalence, $\phi$ is $\underline{V}^{(3)}_0\oplus \underline{V}^{(3)}_3:S^2 \rightarrow G(2,4; \mathbb{C})$ with $K=\frac{2}{3}$ and $\|B\|^2=\frac{8}{3}$;
\\(5) up to $U(6)$ equivalence, $\phi$ is $\widehat{\underline{V}}^{(2)}_0\oplus \underline{V}^{(2)}_2: S^2 \rightarrow G(2,6; \mathbb{C})$ with $K=1$ and $\|B\|^2=2$;
\\(6) up to $U(4)$ equivalence, $\phi$ is  $\underline{V}_{1}^{(2)} \oplus \underline{c}_0: S^2 \rightarrow G(2,4; \mathbb{C})$ with $K=1$ and $\|B\|^2=0$, where $c_0 = (0,0,0,1)^T$;
\\(7) up to $U(6)$ equivalence, $\phi$ is $\underline{V}_{2}^{(4)} \oplus \underline{c}_0: S^2 \rightarrow G(2,6; \mathbb{C})$ with $K=\frac{1}{3}$ and $\|B\|^2=\frac{4}{3}$, where $c_0 = (0,0,0,0,0,1)^T$.
\end{theorem}

\begin{theorem}
Let $\phi: S^2\rightarrow G(2,N;\mathbb{C})$ be a  linearly full conformal minimal immersion with rank $\partial{'}\underline{\phi} = 1$ and  rank $\partial{''}\underline{\phi} = 2$,
  let $K$ and $B$ be its Gauss
 curvature and second fundamental form respectively. If $B$ is
 parallel, then  $\phi$  belongs to one of the following
minimal immersions.
\\(1) up to $U(4)$ equivalence, $\phi$ is  $\underline{V}_{1}^{(3)} \oplus \underline{V}_{3}^{(3)}: S^2 \rightarrow G(2,4; \mathbb{C})$ with $K=\frac{2}{5}$ and $\|B\|^2=0$;
\\(2) up to $U(5)$ equivalence, $\phi$ is $\widehat{\underline{V}}^{(1)}_1\oplus \underline{V}^{(2)}_1: S^2 \rightarrow G(2,5; \mathbb{C})$ with $K=\frac{4}{5}$ and $\|B\|^2=0$;
\\(3) up to $U(6)$ equivalence, $\phi$ is $\underline{V}^{(4)}_3\oplus \underline{\alpha}: S^2 \rightarrow G(2,6; \mathbb{C})$ with $K=\frac{2}{5},  \    \|B\|^2=\frac{4}{5}$, 
where $\alpha = ({V_{2}^{(4)}}^T, \sqrt{48}e^{\sqrt{-1}\theta})^T$ for some constant $\theta$.
\end{theorem}

\begin{theorem}
Let $\phi: S^2\rightarrow G(2,N;\mathbb{C})$ be a $\partial^{'}$-irreducible and $\partial^{''}$-irreducible linearly full conformal minimal immersion,and let $K$ and $B$ be its Gauss
 curvature and second fundamental form respectively. If $B$ is
 parallel, then  $\phi$ belongs to one of the following
minimal immersions.
\\(1) up to $U(5)$ equivalence, $\phi$ is $\underline{V}^{(4)}_1\oplus \underline{V}^{(4)}_3: S^2 \rightarrow G(2,5; \mathbb{C})$  with $K=\frac{1}{5}$ and $\|B\|^2=0$;
\\(2) up to $U(6)$ equivalence, $\phi$ is $\widehat{\underline{V}}^{(2)}_1\oplus \underline{V}^{(2)}_1: S^2 \rightarrow G(2,6; \mathbb{C})$  with $K=\frac{1}{2}$ and $\|B\|^2=0$;
\\(3) up to $U(10)$ equivalence, $\phi$ is $\widehat{\underline{V}}^{(4)}_2\oplus \underline{V}^{(4)}_2: S^2 \rightarrow G(2,10; \mathbb{C})$  with $K=\frac{1}{6}$ and $\|B\|^2=\frac{2}{3}$.
\end{theorem}

In  these four theorems, $\underline{V}^{(n)}_i$ denotes the linearly full Veronese surface in $\mathbb{C}P^n$, its standard expression is given in Section 2 below. If $n\leq N-2$, for one thing, we add zeros to the  end of  $V^{(n)}_i$ such that it belongs to $\mathbb{C}^N$, in the absence of confusion, we also denote it by $V^{(n)}_i$; for another, we denote it as $\widehat{V}^{(n)}_i$ by adding zeros to the front of $V^{(n)}_i$ such that it belongs to $\mathbb{C}^N$ also.

Theorems 1.1-1.4 give the classification of all linearly full conformal minimal immersions from
$S^2$ to $G(2,N; \mathbb{C})$ with parallel second fundamental form, no two of the above
eighteen cases are congruent, i.e. there is no isometry of
$G(2,N;\mathbb{C})$ such that it transforms one case into another.  Furthermore, the Gauss curvatures that occur in Theorems 1.1-1.4 obey Delisle, Hussin and Zakrzewski's conjecture \cite{D-H-Z1, D-H-Z2}.

\section{Preliminaries}\label{sec2}

Let $M, ds_M^2$  be a simply connected domain in the unit sphere $S^2$ with conformal metric $ds_M^2=dzd\overline z$,  and $(z, \overline{z})$ be a complex coordinate on $M$. Denote
$$A_z=\frac{1}{2}s^{-1}\partial s, \quad A_{\overline{z}}=\frac{1}{2}s^{-1}\overline{\partial} s,$$
where $s$ is a smooth map from $M$ to the unitary group $U(N)$, $\partial = \frac {\partial}{\partial z}, \ \overline{\partial}
= \frac {\partial}{\partial \overline{z}}$. Then $s$ is a harmonic map if and only if it satisfies the following equation (cf. \cite{Uhlenbeck}):
\begin{equation}\overline{\partial} A_z=[A_z, A_{\overline{z}}]. \label{eq:2.1}
\end{equation}

 Suppose that $s:S^2 \rightarrow U(N)$ is an isometric immersion, then $s$ is conformal and minimal if it is harmonic. Let $\omega$ be the Maurer-Cartan form on $U(N)$, and let $ds_{U(N)}^2=\frac{1}{8}tr\omega\omega^{*}$ be the metric on $U(N)$. Then the metric induced by $s$ on $S^2$ is locally given by
$$ds^2=-tr A_zA_{\overline{z}}dzd\overline{z}.$$

We consider the complex Grassmann manifold  $G(2,N;\mathbb{C})$ as
the set of Hermitian orthogonal projections from $\mathbb{C}^N$ onto
a $2$-dimensional subspace in $\mathbb{C}^N$. Then map $\phi:
M\rightarrow G(2,N;\mathbb{C})$ is a Hermitian orthogonal projection
onto a $2$-dimensional subbundle $\underline{\phi}$ of the
trivial bundle $\underline{\mathbb{C}}^N = M \times \mathbb{C}^N$
given by setting the fibre $\underline{\phi}_x = \phi (x)$ for
all $x\in M$.  $\underline{\phi}$ is called (a)
\emph{harmonic ((sub-) bundle) } whenever $\phi$ is a harmonic
map. Here $s=\phi-\phi^{\bot}$ is a map from $S^2$ into $U(N)$. It is well known that $\phi$ is harmonic if and only if $s$ is harmonic.
$\phi$ is a \emph{holomorphic (resp. anti-holomorphic) curve}  in $G(2,N;\mathbb{C})$ if and only if $\phi^{\bot}\overline{\partial}\phi=0$ (resp. $\phi^{\bot}{\partial}\phi=0$).

For a conformal minimal immersion $\phi: S^2 \rightarrow
G(2,N;\mathbb{C})$, two harmonic sequences are derived as follows (cf. \cite{Wolfson}):
\begin{equation}\underline{\phi}=\underline{\phi}_{0}\stackrel{\partial{'}}
{\longrightarrow }\underline{\phi}_1\stackrel{\partial{'}}
{\longrightarrow}\cdots\stackrel{\partial{'}}
{\longrightarrow}\underline{\phi}_ {i}\stackrel{\partial{'}}
{\longrightarrow}\cdots\stackrel{\partial{'}}
{\longrightarrow}0, \label{eq:2.2}
\end{equation}
\begin{equation}\underline{\phi}=\underline{\phi}_{0}\stackrel{\partial{''}}
{\longrightarrow}\underline{\phi}_{-1}\stackrel{\partial{''}}
{\longrightarrow}\cdots\stackrel{\partial{''}} {\longrightarrow}
\underline{\phi}_{-i}\stackrel{\partial{''}}{\longrightarrow}\cdots\stackrel{\partial{''}}{\longrightarrow}0,
\label{eq:2.3}
\end{equation} where $\underline{\phi}_{i} = \partial ^\prime
\underline{\phi} _{i-1}$ and $\underline{\phi}_{-i}=
\partial ^{ \prime \prime} \underline{\phi} _{-i+1} $ are Hermitian
orthogonal projections from $S^2 \times \mathbb{C} ^N$ onto
${\underline{Im}}\left(\phi^{\perp}_ {i-1}\partial
\phi_{i-1} \right)$ and
${\underline{Im}}\left(\phi^{\perp}_
{-i+1}\overline{\partial} \phi_ {-i+1}\right)$
respectively, $i=1,2,\ldots$.

Now recall (\cite{Burstall}, \S3A) that a harmonic map $\phi: S^2 \rightarrow
G(2,N;\mathbb{C})$ in \eqref{eq:2.2}(resp. \eqref{eq:2.3}) is said to be
\emph{$\partial^{'}$-irreducible} (resp.
\emph{$\partial^{''}$-irreducible}) if
rank $\underline{\phi}_1$ = rank $\underline{\phi}$ (resp.
rank $\underline{\phi}_{-1}$ = rank $\underline{\phi})$ and \emph{$\partial^{'}$-reducible}
(resp. \emph{$\partial^{''}$-reducible}) otherwise.

For an arbitrary harmonic map $\phi: S^2 \rightarrow
G(2,N;\mathbb{C})$, we know that several consecutive harmonic maps in \eqref{eq:2.2} are not mutually orthogonal generally. So it is meaningful to define the \emph{isotropy order} (cf. \cite{Burstall}) of $\phi$ to be the greatest integer $r$ such that 
$\phi \bot \phi_{i}$ for  $1\leq i \leq r$.

As in \cite{Erdem} call a harmonic map $\phi: S^2 \rightarrow
G(2,N;\mathbb{C})$ \emph{(strongly) isotropic} if $\phi$ has isotropy order  $\geq r$ for all   $r$.
In this case we just set $r= \infty$.

\begin{definition}
\emph{Let $\phi: S^2 \rightarrow G(2,N;\mathbb{C})$ be a map.
$\phi$ is}  linearly full \emph{if $\underline{\phi}$ can not
be contained in any proper trivial subbundle $S^2 \times
\mathbb{C}^n$ of $S^2\times { \mathbb{C}}^N$ ($n< N$)}.
\end{definition}

In this paper, we always assume that $\phi$ is linearly full.

~\

Suppose that $\phi: S^2 \rightarrow G(2,N;\mathbb{C})$ is a
linearly full harmonic map and it belongs to the following harmonic
sequence:
\begin{equation}\underline{\phi}_0\stackrel{\partial{'}}
{\longrightarrow}\cdots \stackrel{\partial{'}} {\longrightarrow}
\underline{\phi}= \underline{\phi}_{i} \stackrel {\partial{'}}
{\longrightarrow} \underline{\phi}_{i+1}\stackrel{\partial{'}}
{\longrightarrow}\cdots \stackrel{\partial{'}} {\longrightarrow}
\underline{\phi}_{i_0}\stackrel{\partial{'}} {\longrightarrow}
0\label{eq:2.11}\end{equation} for some $i = 0, \ldots, i_0$. We choose  local orthonormal  frames $e^{(i)}_1, e^{(i)}_2, \ldots, e^{(i)}_{k_{i}}$
such that they locally span subbundle $\underline{\phi}_{i}$ of
$S^2 \times \mathbb{C}^N$, where $k_i = $ rank
$\underline{\phi}_{i}$.

Let $W_{i}=\left(e^{(i)}_1, e^{(i)}_2, \ldots,
e^{(i)}_{k_{i}}\right)$ be an $\left(N\times k_{i}\right)$-matrix. Then
we have $$\phi_i=W_{i}W^*_{i},\quad
W^*_{i}W_{i} = I_{k_{i}\times k_{i}}, \quad W^*_{i-1}W_{i} =0,\quad
W^*_{i+1} W_{i}=0.$$ By these equations, a
straightforward computation shows that \begin{equation} \left\{
\begin{array}{l} \pa W_{i} = W_{i+1} \Om_{i} +
W_{i} \Psi_{i}, \\
\ov{\pa}W_{i} =-W_{i-1} \Om^*_ {i-1}- W_{i}\Psi^*_{i},
\end{array}\right.  \label{eq:2.4}\end{equation}
where $\Omega_{i}$ is a $\left(k_{i+1} \times k_{i}\right)$-matrix,  $\Psi_{i}$ is a $\left(k_{i}\times k_{i}\right)$-matrix for
$i=0, 1, 2, \ldots, i_0$, and $\Omega_{i_0}=0$.
It is very evident that integrability conditions for  \eqref{eq:2.4} are
$$\overline{\partial}\Omega_i=\Psi_{i+1}^{*}\Omega_i-\Omega_i\Psi_{i}^{*},$$
$$\overline{\partial}\Psi_i+\partial\Psi_{i}^{*}=\Omega_i^*\Omega_i+\Psi_i^*\Psi_i-\Omega_{i-1}\Omega_{i-1}^{*}-\Psi_i\Psi_i^*.$$

For a conformal immersion $\phi: M\rightarrow G(2,N;\mathbb{C})$, we define its \emph{K\"{a}hler angle} to be the function $\theta: M\rightarrow [0, \pi]$ given in terms of a complex coordinate $z$ on $M$ by (cf. \cite{Bolton, Chern})
$$\tan\frac{\theta(p)}{2}=\frac{|d\phi(\partial/\partial \overline{z})|}{|d\phi(\partial/\partial z)|}, \quad p\in M.$$
It is clear that $\theta$ is globally defined and is smooth at $p$ unless $\theta(p)=0$ or $\pi$. $\phi$ is holomorphic (resp. anti-holomorphic) if and only if $\theta(p)=0$ (resp. $\theta(p)=\pi$) for all $p\in M$, while $\phi$ is totally real if and only if $\theta(p)=\frac{\pi}{2}$ for all $p\in M$.

Let $\phi: S^2\rightarrow G(2,N;\mathbb{C})$ be a conformal minimal immersion with the harmonic sequence \eqref{eq:2.11}, put $L_{i} =\textrm{tr} (\Omega_{i} \Omega^*_{i})$, then, in terms of a local complex coordinate $z$, its K\"{a}hler angle $\theta_i$ satisfies
$$(\tan\frac{\theta_i}{2})^2=\frac{L_{i-1}}{L_i}.$$
The metric induced by $\phi$ is
given in the form  \begin{equation}ds^2 _{i} =
(L_{i-1}+L_{i})dzd\overline{z}\triangleq \lambda^2dzd\overline{z}. \label{eq:2.5}\end{equation}
Let $K$ and $B$ be the Gauss curvature and second fundamental form of $\phi$
respectively, then we have \begin{equation} \left\{
\begin{array}{l}
K = -\frac{2}{L_{i-1}+L_{i}}\partial\overline{\partial} \log {(L_{i-1}+L_{i})},
\\ \|B\|^2 = 4\textrm{tr}PP^{\ast},
\end{array}\right.  \label{eq:2.6}\end{equation}
where $P=\partial\left(\frac{A_z}{\lambda^2}\right)$ with $A_{z}=(2\phi-I)\pa \phi$, $I$ is
the identity matrix (cf. \cite{Jiao2008, Jiao2012}).

 In the following, we review the rigidity theorem of conformal
minimal immersions with constant curvature from $S^2$ to
$\mathbb{C}P^N$.

 Let $\psi:S^2 \rightarrow \mathbb{C}P^N$ be a linearly full
conformal minimal immersion, a harmonic sequence is derived as
follows \begin{equation}0\stackrel{\partial^\prime}{\longrightarrow}
\underline{\psi} _0
\stackrel{\partial^\prime}{\longrightarrow}\cdots
\stackrel{\partial^\prime}{\longrightarrow} \underline{\psi}
=\underline{\psi}_i \stackrel{\partial^\prime}{\longrightarrow}
\cdots \stackrel{\partial^\prime}{\longrightarrow} \underline{\psi}
_N \stackrel{\partial^\prime}{\longrightarrow} 0
\label{eq:2.7}\end{equation} for some $i =0, 1, \ldots,N$.

We define a sequence $f_0, \ldots, f_N$ be local sections of
$\underline{\psi}_0, \ldots , \underline{\psi}_N$ inductively such
that $f_0$ is a nowhere zero local section of $\underline{\psi}_0$
(without loss of generality, assume that $\overline{\partial} f_0
\equiv 0$) and $f_{i+1} = \psi^\perp_{i}(
\partial f_{i}) $ for $i=0, \ldots,N-1$. Then we have
some formulae as follows (cf. \cite{Bolton}): \begin{equation}\partial f_i = f_{i+1} +
\frac{\langle\partial f_i,f_i\rangle}{|f_i|^2}f_i, \ i= 0,\ldots, N-1,
\label{eq:2.8}\end{equation}
\begin{equation} \overline{\partial}
f_i= - \frac{|f_i|^2}{|f_{i-1}|^2} f_{i-1}, \ i = 1, \ldots, N.
\label{eq:2.9}\end{equation}
\begin{equation} \partial\overline{\partial}
\log |f_i|^2= l_i-l_{i-1}, \ i = 0, \ldots, N.
\label{2.11}\end{equation}
\begin{equation} \partial\overline{\partial}
\log l_i= l_{i+1}-2l_i+l_{i-1}, \ i = 0, \ldots, N-1,
\label{2.12}\end{equation}
where $l_i=\frac{|f_{i+1}|^2}{|f_{i}|^2}$ for $i=0,...,N-1$, and $l_{-1}=l_N=0$.

Next, we state the definition of degree of a smooth map $\psi$ from a compact Riemann surface $M$ into $G(k, N; \mathbb{C})$ as follows.

\begin{definition}[{\cite{Burstall}}]
\emph{The degree of $\psi$, denoted by $\deg\psi$ is the degree of the induced map $\psi^*: H^2(G(k, N; \mathbb{C}), Z)\cong Z\rightarrow H^2(M, Z)\cong Z$ on second cohomology.} 
\end{definition}

In \eqref{eq:2.7}, let $F_i=f_0\wedge f_1\wedge\cdots\wedge f_i$ be a local lift of the $i$-th osculating curve, where $i=0, \cdots, N$. We write $F_i=f(z)\tilde F_i$, where $f(z)$ is the greatest common divisor of  the $\binom{N+1}{i+1}$ components of $F_i$. Then $\tilde F_i$ is a nowhere zero holomorphic curve, and the degree $\delta^{(N)}_i$ of $F_i$ is given by $\delta^{(N)}_i=1/(2\pi\sqrt{-1})\int_{S^2}\partial\overline{\partial}\log|F_i|^2d\overline z\wedge dz$, which is equal to the degree of the polynomial function $\tilde F_i$. Then we have
$$\delta^{(N)}_i=\frac{1}{2\pi\sqrt{-1}}\int_{S^2}l_i d\overline z\wedge dz.$$

Especially, for harmonic sequence \eqref{eq:2.7}. Let $r(\partial')=$ sum of the indices of the singularities of $\partial'$, which is called the \emph{ramification index} of $\partial'$ by Bolton et al (cf. \cite{Bolton}). Note that if $r(\partial')=0$ in \eqref{eq:2.7} for all $\partial'$, the harmonic sequence is defined \emph{totally unramified} in \cite{Bolton}.
If \eqref{eq:2.7} is a totally unramified harmonic sequence, then (see \cite{Bolton})
\begin{equation} 
\delta^{(N)}_i=(i+1)(N-i).
\label{2.13}\end{equation}

Consider the Veronese sequence
$$0\stackrel{\partial{'}}{\longrightarrow }\underline{V}^{(N)}_0\stackrel{\partial{'}}
{\longrightarrow}\underline{V}^{(N)}_1\stackrel{\partial{'}}
{\longrightarrow}\cdots\stackrel{\partial{'}}
{\longrightarrow}\underline{V}^{(N)}_N\stackrel{\partial{'}}
{\longrightarrow}0.$$ For each $i=0,\ldots, N$,  $\underline{V}^{(N)}_i: S^2\rightarrow \mathbb{C}P^N$ is given by $V^{(N)}_i = (v_{i,0}, \ldots, v_{i,N})^{T}$,  where, for
$z\in S^2$ and $j=0,\ldots, N$,
$$v_{i,j}(z) =
\frac{i!}{(1 +
z\overline{z})^i}\sqrt{\binom{N}{j}}z^{j-i}\sum_{k}(-1)^k
\binom{j}{i-k}\binom{N-j}{k}(z\overline{z})^k.$$
Each map $\underline{V}^{(N)}_i$ satisfies
\begin{equation}
|V^{(N)}_i|^2=\frac{N!i!}{(N-i)!}(1+z\overline{z})^{N-2i},
\label{eq:2.10}\end{equation} it has induced metric $ds^2_i= \frac{N
+ 2i(N - i)}{(1+z\overline{z})^2}dzd\overline{z},$ and the corresponding
constant curvature $K_i$ is given by $K_i= \frac{4}{N +
2i(N-i)}$.

By Calabi's rigidity theorem, Bolton et al proved the following
rigidity result (cf. \cite{Bolton}).
\begin{lemma}[{\cite{Bolton}}]
Let $\psi : S^2 \rightarrow \mathbb{C}P^N$ be a linearly full
conformal minimal immersion of constant curvature. Then, up to a
holomorphic isometry of $\mathbb{C}P^N$, the harmonic sequence
determined by $\psi$ is the Veronese sequence.
\end{lemma}

\section{Holomorphic curves with parallel second fundamental form}

 We recall that an immersion of $S^2$ in $G(2,N;\mathbb{C})$ is conformal and minimal if and only if it is
harmonic. Thus, we shall consider  harmonic maps from $S^2$ to $G(2,N;\mathbb{C})$ with parallel second fundamental form  to give
the proof of Theorems 1.1-1.4 in Section 1.

 Let $\phi: S^2 \rightarrow G(2,N;\mathbb{C})$ be a harmonic map with  Gauss curvature $K$ and second fundamental form $B$.
 Suppose that  $B$ is parallel, it is known that such 2-spheres in $G(2,N;\mathbb{C})$  have constant curvature (cf. 
\cite{Jiao2012}, Theorem 4.5). To give a complete classification, in this paper we analyze $\phi$ by the following six cases:
\\ (I) $\phi$ is a holomorphic curve in $G(2,N;\mathbb{C})$;
\\ (II) rank $\partial{'}\underline{\phi} = $ rank $\partial{''}\underline{\phi} = 1$;
\\ (III) rank $\partial{'}\underline{\phi} = 1$ and rank $\partial{''}\underline{\phi} = 2$;
\\ (IV) rank $\partial{'}\underline{\phi} = $ rank $\partial{''}\underline{\phi} = 2$;
\\ (V) $\phi$ is an anti-holomorphic curve in $G(2,N;\mathbb{C})$;
\\ (VI) rank $\partial{'}\underline{\phi} = 2$ and rank $\partial{''}\underline{\phi} = 1$.

For cases (V) and (VI), since the conjugations of  corresponding $\phi: S^2 \rightarrow G(2,N;\mathbb{C})$ belong to cases (I) and (III) respectively, we only consider the classification of $\phi$ in cases (I)-(IV) here.

In this section we first discuss the case that $\phi$ is holomorphic, then a harmonic sequence is derived by $\phi$ via the $\partial^{'}$-transform
$$0\stackrel{\partial{'}}
{\longrightarrow}\underline{\phi}\stackrel{\partial{'}} {\longrightarrow}
\underline{\phi}_{1}\stackrel{\partial{'}}
{\longrightarrow}\cdots\stackrel{\partial{'}}
{\longrightarrow}0.$$
To characterize $\phi$,  we need the following two Lemmas  about parallel
minimal immersions of 2-spheres in $G(k,N;\mathbb{C})$ as follows:

\begin{lemma}(\cite{Jiao2012})
Let $\phi: S^2\rightarrow G(k,N;\mathbb{C}) $ be a conformal
minimal immersion with Gauss curvature $K$ and second fundamental form $B$. Suppose that
$B$ is parallel, then the following equations
\begin{equation} \left\{
\begin{array} {l}   \lambda^2\left(2K +
\|B\|^2\right)A_{\overline{z}} + 4[A_{\overline{z}},
[A_z,A_{\overline{z}}]] =
 0,\\
  \lambda^2 \left(\frac{\|B\|^2}{4} - K\right)P +
[[A_{\overline{z}},A_z], P] = 0
\end{array} \right. \label{eq:3.1}\end{equation} hold.
\end{lemma}

\begin{lemma} (\cite{Jiao2012})
Let $\phi: S^2\rightarrow G(k,N;\mathbb{C}) $ be a conformal
minimal immersion with Gauss curvature $K$ and second fundamental form $B$. Then $B$ is
parallel if and only if the equation
\begin{equation}\frac{\lambda^2}{16}\|B\|^2(8K + \|B\|^2) - 2 \textrm{tr}[A_z,
P][A_{\overline{z}}, P^\ast] + 5\textrm{tr}[A_z,
A_{\overline{z}}][P, P^\ast] = 0 \label{eq:3.2}\end{equation} holds.
\end{lemma}

In the following we shall analyze $\phi$ by rank $\underline{\phi}_{1}$=1 and rank $\underline{\phi}_{1}$=2 respectively.

\subsection{$\partial^{'}$-reducible holomorphic curves with parallel second fundamental form}

Here we suppose that $\phi$ is a holomorphic curve from $S^2$ to $G(2,N;\mathbb{C})$ with parallel second fundamental form
and rank $\partial^{'}\underline{\phi} =1$. To characterize these holomorphic curves, firstly by denoting $\partial^{(-1)}g=\partial^{''}g, \ \partial^{(-i-1)}g=\partial^{''}(\partial^{(-i)}g)$, we state one of Burstall and Wood' results as follows:

\begin{lemma}[{\cite{Burstall}}]
Let $\underline{\phi}:S^2\rightarrow G(2,N;\mathbb{C})$ be harmonic with $\partial{'}\underline{\phi}$ of rank one and $A''_{\phi}(\underline{ker}A^{'\perp}_{\phi})=0$. Then either (i) there is an anti-holomorphic map $g:S^2\rightarrow \mathbb{C}P^{N-1}$ and $\underline{\phi}=\partial^{(-i)}\underline{g}\oplus \partial^{(-i-1)}\underline{g}$ for some integer $i\geq 0$, (it can be shown that $\phi$ is a Frenet pair); or (ii) there are maps $g,h:S^2 \rightarrow \mathbb{C}P^{N-1}$ anti-holomorphic and holomorphic respectively such that $\partial'\underline{h}\perp \underline{g}$ and $\underline{\phi}=\underline{g}\oplus \underline{h}$, i.e. $\underline{\phi}$ is a mixed pair.
\end{lemma}

$A^{'}_{\phi}$ and $A^{''}_{\phi}$ shown in Lemma 3.3 are vector bundle morphisms from $\underline{\phi}$ to $\underline{\phi}^{\bot}$, they are defined by $A^{'}_{\phi}(v)=\pi_{\phi^{\bot}}(\partial v)$ and $A^{''}_{\phi}(v)=\pi_{\phi^{\bot}}(\overline{\partial} v)$ respectively for some $v\in\mathbb{C}^{\infty}(\underline{\phi})$ (cf. \cite{Burstall}). Here clearly we have $A''_{\phi}(\underline{ker}A^{'\perp}_{\phi})=0$ from the assumption that $\phi$ is holomorphic. With the help of Lemma 3.3, we now consider two cases: (1) $\phi$ is a Frenet pair, and (2) $\phi$ is a mixed pair.

{\bfseries \emph{Firstly we consider the case that $\phi$ is a Frenet pair.}}
 In this case $$\underline{\phi}=\underline{f}_0 \oplus \underline{f}_1:S^2\rightarrow G(2,N;\mathbb{C}),$$ it belongs to the harmonic sequence as follows
\begin{equation}0\stackrel{\partial{'}}
{\longrightarrow}\underline{\phi}=\underline{f}_0 \oplus \underline{f}_1\stackrel{\partial{'}} {\longrightarrow}
\underline{f}_{2}\stackrel{\partial{'}}
{\longrightarrow}\cdots\stackrel{\partial{'}}
{\longrightarrow}\underline{f}_{n}\stackrel{\partial{'}}
{\longrightarrow}0, \label{3.3}\end{equation}
where \begin{equation}0\stackrel{\partial{'}}
{\longrightarrow} \underline{f}_0\stackrel{\partial{'}}
{\longrightarrow}\underline{f}_1\stackrel{\partial{'}}
{\longrightarrow} \cdots \stackrel {\partial{'}}
{\longrightarrow}\underline{f}_n \stackrel{\partial{'}}
{\longrightarrow}0  \label{3.4}\end{equation}   is a linearly full harmonic sequence in
$\mathbb{C}P^{n}$ with $N=n+1$. By making use of  $\phi=\frac{f_0f_0^{*}}{|f_0|^2}+\frac{f_1f_1^{*}}{|f_1|^2}$, we  get 
\begin{equation}\partial\phi = \frac{f_2f_1^{*}}{|f_1|^2}, \quad \overline{\partial}\phi = \frac{f_1f_2^{*}}{|f_1|^2}, \quad \lambda^2=\frac{|f_2|^2}{|f_1|^2},  \quad      A_z = -\frac{f_2f_1^{*}}{|f_1|^2},  \quad  A_{\overline{z}} = \frac{f_1f_2^{*}}{|f_1|^2},   \label{eq:3.3}\end{equation}
$$[A_z, A_{\overline{z}}]  = \frac{|f_2|^2}{|f_1|^4}f_1f_1^{*}-\frac{1}{|f_1|^2}f_2f_2^{*}, \quad [A_{\overline{z}}, [A_z, A_{\overline{z}}]]  =  -\frac{2|f_2|^2}{|f_1|^4}f_1f_2^{*}.$$
Then it follows from  $P=\partial\left(\frac{A_z}{\lambda^2}\right)$ that 
\begin{equation}P = -\frac{1}{|f_2|^2}f_3f_1^{*}+\frac{|f_1|^2}{|f_0|^2|f_2|^2}f_2f_0^{*}\label{eq:3.4}\end{equation} and 
\begin{equation}[[A_{\overline{z}}, A_z], P]=\frac{1}{|f_0|^2}f_2f_0^{*}-\frac{1}{|f_1|^2}f_3f_1^{*}.\label{eq:3.5}\end{equation}
For convenience, we denote
\begin{equation} M_1=-\frac{\lambda^2(2K+\|B\|^2)}{4},    \quad    M_2=\lambda^2(K-\frac{\|B\|^2}{4}).\label{eq:3.6}\end{equation}   From the two equations of  \eqref{eq:3.1}   we get 
$$M_1=-2\lambda^2,  \quad   M_2=\lambda^2,$$ which  verifies that
$$K+\frac{\|B\|^2}{2}=4,  \quad  K-\frac{\|B\|^2}{4}=1,$$ and therefore $$ K=2, \quad \|B\|^2=4.$$ 

Since the second fundamental form of the map $\phi$ is paralle, its Gauss curvature is a constant (cf. \cite{Jiao2012}, Theorem 4.5).
Hence by (\cite{HP2}, Lemma 4.1) we know that $\underline{f}_0: S^2\rightarrow \mathbb{C}P^2$ is of constant curvature, then harmonic sequences \eqref{3.3} and \eqref{3.4} are both totally unramified. From \eqref{2.13} we get
$$\delta^{(n)}_0=n,\ \delta^{(n)}_1=2(n-1), \ \delta^{(n)}_2=3(n-2).$$
By substituting the metric of  $\phi$ shown in  \eqref{eq:3.3} into  \eqref{eq:2.6}, we have
$$2=K=-\frac{2}{\lambda^2}\partial\overline{\partial}\log \lambda^2=-\frac{2}{l_1}\partial\overline{\partial}\log l_1=4-2\frac{l_0+l_2}{l_1}=4-2\frac{\delta^{(n)}_0+\delta^{(n)}_2}{\delta^{(n)}_1},$$
then
 $$n=2.$$
Using the rigidity
theorem of Bolton et al (\cite{Bolton}), up to a holomorphic isometry of
$\mathbb{C}P^{2}$,  $\underline{f}_0$ is a Veronese surface. We
can choose a complex coordinate $z$ on
$\mathbb{C}=S^2\backslash\{pt\}$ so that $ {f}_0 = {U}{V}^{(2)}_0$,
where $U\in U(3)$ and ${V}^{(2)}_0$ has the standard expression
given in Section 2.
By Lemma 3.2, it can easily be checked that, for any $U\in U(3)$,  \begin{equation}\underline{\phi}=\underline{UV}^{(2)}_0\oplus \underline{UV}^{(2)}_1: S^2\rightarrow G(2,3;\mathbb{C})\label{eq:3.7}\end{equation} is of parallel second fundamental form because it satisfies \eqref{eq:3.2}.

{\bfseries \emph{Next, we consider the case that $\phi$ is a  mixed pair.}} In this case $$\underline{\phi}=\underline{f}_0 \oplus \underline{c}_0: S^2\rightarrow G(2,N;\mathbb{C}),$$
where $\underline{c}_0$ is the line bundle spanned by constant vector $(0,0,\ldots, 0,1)^T$ in $\mathbb{C}^{N}$,
$\phi$ belongs to the harmonic sequence as follows
$$0\stackrel{\partial{'}}
{\longrightarrow}\underline{\phi}=\underline{f}_0 \oplus \underline{c}_0\stackrel{\partial{'}} {\longrightarrow}
\underline{f}_{1}\stackrel{\partial{'}}
{\longrightarrow}\cdots\stackrel{\partial{'}}
{\longrightarrow}\underline{f}_{n}\stackrel{\partial{'}}
{\longrightarrow}0, \ N=n+2.$$
By making use of  $\phi=\frac{f_0f_0^{*}}{|f_0|^2}+\frac{c_0c_0^{*}}{|c_0|^2}$ and a similar calculation as the first case above,  it is very evident that
\begin{equation}\lambda^2=\frac{|f_1|^2}{|f_0|^2}.\label{eq:3.8}\end{equation}
Thus  $\underline{f}_0: S^2\rightarrow \mathbb{C}P^n, \ n=N-2$  is also of constant curvature,  and there exists some $U\in U(n+1)$ s.t. $f_0=UV^{(n)}_0$.
To determine $\phi$, we shall divide our discussion into two cases, according as $n=1$,  or $n\geq 2$.

 If $n=1$.  Under this supposition, it can be checked that for any $U\in U(3)$, 
 \begin{equation}\underline{\phi}=\underline{UV}_0^{(1)} \oplus \underline{c}_0\label{eq:3.9}\end{equation}
  is a totally geodesic map from $S^2$ to $G(2,3;\mathbb{C})$ with constant curvature $K=4$ by direct computation (adding zero to the end of  $V^{(1)}_0$ such that it belongs to $\mathbb{C}^3$, in the absence of confusion, we also denote it by $V^{(1)}_0$).

 If $n\geq 2$.  Direct computations show that 
\begin{equation} A_{\overline{z}}=\frac{f_0f_1^{*}}{|f_0|^2},   \ [A_{\overline{z}}, [A_z, A_{\overline{z}}]]  =  -\frac{2|f_1|^2}{|f_0|^4}f_0f_1^{*},  \   P = -\frac{f_2f_0^{*}}{|f_1|^2},  \ [[A_{\overline{z}}, A_z], P]=-\frac{f_2f_0^{*}}{|f_0|^2}.\label{eq:3.10}\end{equation}
Then using relations $[A_{\overline{z}}, [A_z, A_{\overline{z}}]]=M_1A_{\overline{z}}$ and $[[A_{\overline{z}}, A_z], P]=M_2P$, we find easily
\begin{equation}M_1=-2\lambda^2,   \quad M_2=\lambda^2. \end{equation} 
 With it we further obtain \begin{equation} K=2, \quad \|B\|^2=4,\label{eq:3.12}\end{equation}
and then
$$2=K=-\frac{2}{\lambda^2}\partial\overline{\partial}\log \lambda^2=-\frac{2}{l_0}\partial\overline{\partial}\log l_0=4-2\frac{l_1}{l_0}=4-2\frac{\delta^{(n)}_1}{\delta^{(n)}_0}.$$
This together with $\delta^{(n)}_0=n, \ \delta^{(n)}_1=2(n-1)$
 implies $$n=2.$$
From Lemma 2.2, up to a holomorphic isometry of $\mathbb{C}P^2$, $\underline{f}_0, \ \underline{f}_1, \ \underline{f}_2: S^2\rightarrow \mathbb{C}P^2$ are Veronese surfaces.
Then by Lemma 3.2, it can easily be checked that,  for any $U\in U(4)$ \begin{equation}\underline{\phi}=\underline{UV}^{(2)}_0\oplus \underline{c}_0: S^2\rightarrow G(2,4;\mathbb{C})\label{eq:3.13}\end{equation} is of parallel second fundamental form with constant curvature $K=2$ (adding zero to the end of  $V^{(2)}_0$ such that it belongs to $\mathbb{C}^4$).

Summing up, we get the following property

\begin{prop}
Let $\phi: S^2\rightarrow G(2,N;\mathbb{C})$ be a linearly full holomorphic curve with parallel  second fundamental form. Suppose rank $\partial^{'}\underline{\phi} =1$,  then $\phi$ is congruent to cases (1) (2) or (3) in Theorem 1.1.
\end{prop}

\subsection{$\partial^{'}$-irreducible  holomorphic curves with parallel second fundamental form}

Let $\phi: S^2 \rightarrow G(2,N;\mathbb{C})$ be a linearly full conformal minimal immersion with  Gauss curvature $K$ and second fundamental form $B$. Suppose that  $B$ is parallel and $\phi$ is holomorphic with $\partial^{'}$-irreducible, then a harmonic sequence derived by $\phi$ via the $\partial^{'}$-transform is  as follows
\begin{equation}0\stackrel{\partial{'}}
{\longrightarrow}\underline{\phi}\stackrel{\partial{'}} {\longrightarrow}
\underline{\phi}_{1}\stackrel{\partial{'}}
{\longrightarrow}\cdots\stackrel{\partial{'}}
{\longrightarrow}0\label{eq:3.14}\end{equation} with rank $\underline{\phi}_{1}$ =2.

Since $\phi$ is holomorphic, it is possible to  obtain its local section $f_0$ such that $\underline{f}_0$ is a holomorphic subbundle of $\underline{\phi}$. Without loss of generality, we assume
that $\overline{\partial} f_0= 0$. Therefore $\underline{f}_0$ is a linearly full harmonic map from $S^2$ to $\mathbb{C}P^n$ for some $n < N$ and belongs to the following harmonic sequence
$$0\stackrel{\partial{'}}
{\longrightarrow}\underline{f}_0\stackrel{\partial{'}} {\longrightarrow}
\underline{f}_{1}\stackrel{\partial{'}}
{\longrightarrow}\cdots \stackrel{\partial{'}} {\longrightarrow}
\underline{f}_{n}\stackrel{\partial{'}}
{\longrightarrow}0.$$

 From the fact rank $\underline{\phi}$=2, we immediately see that there exists another local section $\alpha$ of $\underline{\phi}$ such that
$\underline{\phi}=\underline{\alpha}\oplus \underline{f}_0$.  Set $$\alpha_1=\partial \alpha-\frac{\langle \partial \alpha, \alpha\rangle}{|\alpha|^2}\alpha, \ \alpha_{-1}=\overline{\partial} \alpha-\frac{\langle \overline{\partial} \alpha, \alpha\rangle}{|\alpha|^2}\alpha,   \ \beta=A^{'}_{\phi}(f_0), \ \underline{\gamma}=\underline{\beta}^{\bot}\cap \underline{\phi}_{1}.$$
By Theorem 2.4 of \cite{Burstall} and \eqref{eq:3.14},  we have a new harmonic map $$\underline{\alpha}\oplus \underline{\beta}: S^2 \rightarrow G(2,N; \mathbb{C}),$$  which
belongs to the following harmonic sequence:
$$0\stackrel{\partial{''}}
{\longleftarrow}\underline{f}_0\stackrel{\partial{''}}
{\longleftarrow}\underline{\alpha}\oplus \underline{\beta}\stackrel{\partial{'}} {\longrightarrow}
\cdots\stackrel{\partial{'}}
{\longrightarrow}0.$$
 From it we  arrive at the following equations

 $$\alpha_{1}=\frac{\langle \partial \alpha, \beta\rangle}{|\beta|^2}\beta+\frac{\langle \partial \alpha, \gamma\rangle}{|\gamma|^2}\gamma, \ \alpha_{-1}=-\frac{\langle \alpha, f_{1}\rangle}{|f_{0}|^2}f_0.$$
By making use of  $\phi=\frac{\alpha\alpha^{*}}{|\alpha|^2}+\frac{f_0f_0^{*}}{|f_0|^2}$,  it is an elementary exercise  to show that
$$ A_{\overline{z}} = \frac{\langle  \beta, f_{1}\rangle}{|f_{0}|^2|\beta|^2} f_0\beta^{*} +
\frac{1}{|\alpha|^2}\alpha\alpha_{1}^{*},     \quad    \lambda^2=\frac{\langle \beta, f_{1}\rangle\langle f_{1}, \beta\rangle}{|f_{0}|^2|\beta|^2}+\frac{|\alpha_{1}|^2}{|\alpha|^2},$$

$$
\begin{array}{lll} [A_{\overline{z}}, [A_z, A_{\overline{z}}]] & = &-2\frac{\langle f_{1}, \beta\rangle\langle \beta, f_{1}\rangle\langle \beta, f_{1}\rangle}{|f_{0}|^4|\beta|^4} f_0\beta^{*}-
2\frac{\langle f_{1}, \beta\rangle\langle \beta, f_{1}\rangle\langle  \beta, \alpha_{1}\rangle}{|f_{0}|^2|\alpha|^2|\beta|^4}\alpha\beta^{*}\\
& &-2\frac{\langle \beta,  f_{1} \rangle\langle \alpha_{1}, \beta\rangle}{|f_{0}|^2|\alpha|^2|\beta|^2}f_0\alpha_{1}^{*}
-2\frac{|\alpha_{1}|^2}{|\alpha|^4}\alpha\alpha_{1}^{*}.\end{array} $$

From the supposition that $\phi$ is of parallel second fundamental form, then the following can be easily checked
$$\left\{
\begin{array}{l} M_1\beta = -2\frac{\langle f_{1}, \beta\rangle\langle \beta, f_{1}\rangle}{|f_{0}|^2|\beta|^2}\beta-2\frac{\langle \beta, \alpha_{1}\rangle}{|\alpha|^2}\alpha_{1}, \\
M_1\alpha_{1} = -2\frac{\langle f_{1}, \beta\rangle\langle \beta, f_{1}\rangle\langle  \alpha_{1}, \beta\rangle}{|f_{0}|^2|\beta|^4}\beta-2\frac{|\alpha_{1}|^2}{|\alpha|^2}\alpha_{1}
\end{array}\right.$$ by the relation $[A_{\overline{z}}, [A_z, A_{\overline{z}}]]=M_1A_{\overline{z}}$ given in Lemma 3.1. From it straightforward computations show
\begin{equation}\langle \alpha_{1}, \beta\rangle=0, \ -M_1=2\frac{\langle \beta, f_{1}\rangle\langle f_{1}, \beta\rangle}{|f_{0}|^2|\beta|^2}=2\frac{|\alpha_{1}|^2}{|\alpha|^2}=\lambda^2,\label{eq:3.15}\end{equation}
which establishes that \begin{equation}K+\frac{\|B\|^2}{2}=2.\label{eq:3.16}\end{equation}
Set $\lambda_1=-\frac{\langle f_{1}, \beta\rangle}{\lambda^2|f_{0}|^2|\beta|^2}, \ \lambda_2=-\frac{1}{\lambda^2|\alpha|^2}$, applying  equation $P=\partial(\frac{A_z}{\lambda^2})$ it follows that
\begin{equation}
 P = \partial \lambda_1  \beta f_0^{*} + \lambda_1\partial\beta f_0^{*} + \partial \lambda_2 \alpha_{1}\alpha^{*}+ \lambda_2\partial\alpha_1\alpha^{*}+\lambda_2\alpha_1(\overline{\partial}\alpha)^{*}. \label{eq:3.17}\end{equation}
Then under the assumption $\nabla B=0$ and using Lemma 3.1, $[[A_{\overline{z}}, A_z], P]=M_2P$ is equivalent to
\begin{equation}
\begin{array}{lll} \lambda_1(M_2-\frac{|\alpha_{1}|^2}{|\alpha|^2})\partial\beta & = &[2\partial\lambda_{1}\frac{|\alpha_{1}|^2}{|\alpha|^2}+\lambda_1\frac{|\alpha_{1}|^2\langle \partial\beta, \beta\rangle}{|\alpha|^2|\beta|^2}-M_2\partial\lambda_{1}]\beta \\
& &+[\lambda_1\frac{\langle \partial\beta, \alpha_{1}\rangle}{|\alpha|^2}-2\lambda_2\frac{|\alpha_{1}|^2\langle f_{1}, \alpha\rangle}{|\alpha|^2|f_{0}|^2}+\lambda_2M_2\frac{\langle f_{1}, \alpha\rangle}{|f_{0}|^2}]\alpha_{1},\end{array} \label{eq:3.18}\end{equation}
and
\begin{equation}
\begin{array}{lll} \lambda_2(M_2-\frac{|\alpha_{1}|^2}{|\alpha|^2})\partial \alpha_{1} & = &[2\partial\lambda_{2}\frac{|\alpha_{1}|^2}{|\alpha|^2}+2\lambda_2\frac{|\alpha_{1}|^2\langle \alpha, \overline{\partial} \alpha\rangle}{|\alpha|^4}+\lambda_2\frac{\langle \partial \alpha_{1}, \alpha_{1}\rangle}{|\alpha|^2}\\
& &-M_2\partial\lambda_{2}-M_2\lambda_2\frac{\langle \alpha, \overline{\partial} \alpha\rangle}{|\alpha|^2}]\alpha_{1}\\
& &+\lambda_2\frac{|\alpha_{1}|^2\langle \partial\alpha_1, \beta\rangle}{|\alpha|^2|\beta|^2}\beta.\end{array} \label{eq:3.19}\end{equation}

With this, we have therefore conclude

\begin{prop}
Let $\phi: S^2\rightarrow G(2,N;\mathbb{C})$ be a  linearly full $\partial^{'}$-irreducible  holomorphic curve with parallel  second fundamental form, then $\phi$  is  congruent to cases (4) or (5) in Theorem 1.1.
\end{prop}
\begin{proof}
In our prove, by observing equations   \eqref{eq:3.18} and \eqref{eq:3.19},   we discuss the rigidity of $\phi$ by cases $M_2=\frac{|\alpha_{1}|^2}{|\alpha|^2}$ and
 $M_2 \neq \frac{|\alpha_{1}|^2}{|\alpha|^2}$ respectively.

\emph{Firstly we discuss the case $M_2=\frac{|\alpha_{1}|^2}{|\alpha|^2}$.}
In this case we immediately have $$M_2=\frac{\lambda^2}{2},$$ which shows  $$K=1, \ \|B\|^2=2$$ by combining it with \eqref{eq:3.16}. Therefore it follows from \eqref{eq:3.18} and \eqref{eq:3.19} that
\begin{equation} \langle \partial\alpha_1,   \beta\rangle =0, \ \partial\lambda_1+\lambda_1\frac{\langle  \partial \beta, \beta\rangle}{|\beta|^2} =0,\label{eq:3.20}\end{equation}
\begin{equation}\partial\lambda_2+\lambda_2\frac{\langle  \partial \alpha_1, \alpha_1\rangle}{|\alpha_1|^2} + \lambda_2\frac{\langle \alpha, \overline{\partial} \alpha \rangle}{|\alpha|^2}=0,\label{eq:3.21}\end{equation}
\begin{equation}\lambda_1\frac{\langle \partial\beta, \alpha_1 \rangle}{|\alpha_{1}|^2}-\lambda_2\frac{\langle f_{1}, \alpha\rangle}{|f_{0}|^2}=0.\label{eq:3.22}\end{equation}
Choose local frame
$$e_1 = \frac{f_0}{|f_0|},  \  e_2 = \frac{\alpha}{|\alpha|}, \
e_3 = \frac{\beta}{|\beta|},  \ e_4 = \frac{\alpha_{1}}{|\alpha_{1}|},$$ by the first relation in  \eqref{eq:3.15} we know that this frame is  unitary.
Set $W_0 = (e_1, e_2), \quad W_1 = (e_3, e_4)$, then by \eqref{eq:2.4}, we obtain $$\Omega_0 = \left(\begin{array}{ccccccc}
 \frac{\langle  f_{1}, \beta\rangle}{|f_{0}||\beta|} & 0 \\
0  & \frac{|\alpha_{1}|}{|\alpha|} \end{array}\right).$$
This together with  equation  \eqref{eq:3.15} implies that
\begin{equation}|\det \Omega_0|^2 = \frac{(\lambda^2)^2}{4}. \label{eq:3.23}\end{equation}

Since $\phi: S^2 \rightarrow G(2,N;\mathbb{C})$ is a harmonic map with constant curvature $K=1$,   complex coordinate $z$
on $\mathbb{C} =S^2 \backslash \{pt\}$ can be chosen so that  the induced  metric
$ds^2 = \lambda^2dz d{\overline{z}}$ is given by
$ds^2 =
\frac{4}{(1+z\overline{z})^2} dz d{\overline{z}}$, where
$L_0
=\frac{4}{(1+z\overline{z})^2}$. 
Therefore  we clearly get
$$L_1 = L_0 $$
from the unintegrated Pl\"{u}cker formulae (cf. \cite{Jiao2008})
$$\partial \overline{\partial} \log|\det \Omega_0|^2 = - 2
L_{0} + L_{1}$$ and \eqref{eq:3.23},
which implies that $\phi_1:S^2\rightarrow G(2,N;\mathbb{C})$ is totally real with constant curvature $\frac{1}{2}$. Then from \cite{Berndt, HP2}  and \cite{Q4}, adding zeros to the end of  $V^{(2)}_1$ such that it belongs to $\mathbb{C}^6$,
$\underline{\phi}_{1}=\overline{\underline{UV}}^{(2)}_1\oplus \underline{UV}^{(2)}_1:S^2 \rightarrow G(2,6;\mathbb{C})$ or  $\underline{\phi}_{1}=\textbf{\underline{J}}\underline{UV}^{(2)}_1\oplus \underline{UV}^{(2)}_1:S^2 \rightarrow G(2,6;\mathbb{C})$ for some $U\in U(6)$, so we have
$$\underline{\phi}=\overline{\underline{UV}}^{(2)}_2\oplus \underline{UV}^{(2)}_0: S^2 \rightarrow G(2,6;\mathbb{C})$$ or  $$\underline{\phi}=\textbf{\underline{J}}\underline{UV}^{(2)}_2\oplus \underline{UV}^{(2)}_0:S^2 \rightarrow G(2,6;\mathbb{C}),$$
here $\textbf{J}: \mathbb{C}^6\rightarrow \mathbb{C}^6$ is the conjugate linear map given by left multiplication by $j$ (cf. \cite{HP2}).
In the absence of confusion, let $V^{(2)}_0=(1, \ \sqrt{2}z, \ z^2, \ 0, \ 0, \ 0)^T$ and $\widehat{V}^{(2)}_0=(0, \ 0, \ 0, \ 1, \ \sqrt{2}z, \ z^2)^T$ (we shall use such notions repeated below),   in summary, up to an isometry of $G(2,N; \mathbb{C})$, $\phi$ can be expressed as
$$\underline{\phi}=\underline{U}\underline{\widehat{V}}^{(2)}_0\oplus \underline{UV}^{(2)}_0: S^2 \rightarrow G(2,6; \mathbb{C}),$$ which is of parallel second fundamental form with $K=1$ and $\|B\|^2=2$, and it is congruent to the case (5) in Theorem 1.1.

\emph{Next we discuss the case $M_2\neq\frac{|\alpha_{1}|^2}{|\alpha|^2}$.} In this case, using Theorem 4.5 of \cite{Jiao2012}, we get $$\partial\alpha_{1}=\frac{\langle \partial \alpha_{1}, \alpha_{1} \rangle}{|\alpha_{1}|^2}\alpha_{1}+\frac{\langle \partial \alpha_{1},\beta \rangle}{|\beta|^2}\beta, \ \partial\beta =\frac{\langle \partial \beta, \alpha_{1} \rangle}{|\alpha_{1}|^2}\alpha_{1} + \frac{\langle \partial \beta, \beta \rangle}{|\beta|^2}\beta,$$ and
relations \eqref{eq:3.20}-\eqref{eq:3.22}
from the facts $M_2\neq \frac{|\alpha_{1}|^2}{|\alpha|^2}, \ M_2\neq \lambda^2$ and conditions shown in \eqref{eq:3.18} and \eqref{eq:3.19}.

Then substituting \eqref{eq:3.20}-\eqref{eq:3.22} into \eqref{eq:3.17}, we get $P=0$, i.e. $\phi: S^2\rightarrow G(2,4;\mathbb{C})$ is totally geodesic  with constant curvature $K=2$ from \eqref{eq:3.16}, the harmonic sequence given in \eqref{eq:3.14}  becomes

$$0\stackrel{\partial{'}}
{\longrightarrow}\underline{\phi}\stackrel{\partial{'}} {\longrightarrow}
\underline{\phi}_{1}\stackrel{\partial{'}}
{\longrightarrow}0.$$
Adding zeros to $V^{(1)}_0$,  let $V^{(1)}_0=(1, \ z,  \ 0, \ 0)^T$ and $\widehat{V}^{(1)}_0=(0, \ 0,  \ 1, \ z)^T$, it follows from  \cite{Chi} and \cite{Fei} that, up to an isometry of $G(2,4;\mathbb{C})$,
$$\underline{\phi}=\underline{U}\underline{\widehat{V}}^{(1)}_0\oplus \underline{UV}^{(1)}_0: S^2 \rightarrow G(2,4; \mathbb{C})$$ is totally geodesic with $K=2$ for some $U\in U(4)$, and it is congruent to the case (4) in Theorem 1.1. In summary we get the conclusion.
\end{proof}

Propositions 3.4 and 3.5 give a classification of all linearly full holomorphic  maps from $S^2$ to $G(2,N;\mathbb{ C})$ with parallel second fundamental form, from them we get Theorem 1.1 in Section 1.

\section{Minimal two-spheres with parallel second fundamental form and rank $\partial{'}\underline{\phi} = $ rank $\partial{''}\underline{\phi} = 1$}

Accordingly, in this section, we consider conformal minimal immersions $\phi: S^2 \rightarrow G(2,N;\mathbb{C})$ under the assumption that $\nabla B=0$ and rank $\partial{'}\underline{\phi} = $ rank $\partial{''}\underline{\phi} = 1$.
To characterize $\phi$,  we first prove the following property:

\begin{prop}
Let $\phi: S^2\rightarrow G(2,N;\mathbb{C})$ be a linearly full conformal minimal immersion with the second fundamental form $B$. Suppose
that $B$ is parallel and rank $\partial{'}\underline{\phi} = $ rank $\partial{''}\underline{\phi} = 1$, then $\phi$ belongs to one of the following
cases.
\\ (i) $\phi$ is a Frenet pair, i.e. $\underline{\phi} = \underline{f}_{i-1} \oplus
\underline{f}_i $, where $\underline{f}_i:S^2\rightarrow \mathbb{C}P^{n}$ is harmonic and $N =
 n+1$;
 \\ (ii) $\phi$ is a  mixed pair, i.e. $\underline{\phi} ={\underline{f}}_0\oplus
\underline{g}_m$, where $\underline{f}_0: S^2\rightarrow \mathbb{C}P^{n}$ is
holomorphic and $\underline{g}_m: S^2\rightarrow \mathbb{C}P^{m}$ is
anti-holomorphic;
 \\ (iii) $\underline{\phi} = \underline{f}_i \oplus
\underline{c}_0 $, where $c_0=(0,0, \ldots, 0, 1)^T$ in $\mathbb{C}^N$, $\underline{f}_i:
S^2\rightarrow \mathbb{C}P^{n}$ is harmonic and $N =
 n+2$.
\end{prop}
\begin{proof}
The harmonic sequence  derived by $\phi$ via the $\partial^{'}$ and $\partial^{''}$-transforms is as follows:
\begin{equation}0 \stackrel{\partial{''}}
{\longleftarrow}\cdots\stackrel{\partial{''}} {\longleftarrow}
\underline{\phi}_{-1}\stackrel{\partial{''}}
{\longleftarrow} \underline{\phi}\stackrel {\partial{'}}
{\longrightarrow}\underline{\phi}_1  \stackrel{\partial{'}}
{\longrightarrow}\cdots \stackrel{\partial{'}}
{\longrightarrow}0. \label{eq:4.1}\end{equation}
Since $\phi$ is harmonic and rank $\partial{'}\underline{\phi} = $ rank $\partial{''}\underline{\phi} = 1$, it is possible to choose local sections $f_{i+1}, \ g_{j-1}$ of $\underline{\phi}_1$ and $\underline{\phi}_{-1}$ respectively such that $\underline{\phi}_1=\underline{f}_{i+1}, \ \underline{\phi}_{-1}=\underline{g}_{j-1}$, here $\underline{f}_{i+1}:S^2\rightarrow \mathbb{C}P^n$ and $\underline{g}_{j-1}:S^2\rightarrow \mathbb{C}P^m$ are both harmonic. Such \eqref{eq:4.1} can be rewritten as
\begin{equation}0\stackrel{\partial{''}}
{\longleftarrow}\underline{g}_0\stackrel{\partial{''}}
{\longleftarrow}\cdots\stackrel{\partial{''}}
{\longleftarrow}\underline{g}_{j-1}\stackrel{\partial{''}}
{\longleftarrow}\underline{\phi}\stackrel{\partial{'}} {\longrightarrow}
\underline{f}_{i+1}\stackrel{\partial{'}}
{\longrightarrow}\cdots\stackrel{\partial{'}} {\longrightarrow}
\underline{f}_{n}\stackrel{\partial{'}} {\longrightarrow}
0.\label{eq:4.2}\end{equation}
In \eqref{eq:4.2},  $g_j$ and $f_i$ are both local sections of $\underline{\phi}$.

If $\underline{g}_j = \underline{f}_i$.  Let
$\rho$ be the local section of $\underline{\phi}$ such that
$\underline{\phi}=\underline{f}_i\oplus\underline{\rho}$, then \eqref{eq:4.2}   becomes
$$0\stackrel{\partial{''}}
{\longleftarrow}\underline{f}_0\stackrel{\partial{''}}
{\longleftarrow}\cdots\stackrel{\partial{''}}
{\longleftarrow}\underline{f}_{i-1}\stackrel{\partial{''}}
{\longleftarrow}\underline{\phi}=\underline{f}_i\oplus\underline{\rho}\stackrel{\partial{'}} {\longrightarrow}
\underline{f}_{i+1}\stackrel{\partial{'}}
{\longrightarrow}\cdots\stackrel{\partial{'}} {\longrightarrow}
\underline{f}_{n}\stackrel{\partial{'}} {\longrightarrow}
0.$$
 By using  of $\langle\rho, f_{i-1}\rangle=\langle\rho, f_i\rangle=\langle\rho, f_{i+1}\rangle=0$, we get 
$\overline{\partial} \rho=\frac{\langle \overline{\partial} \rho, \rho\rangle}{|\rho|^2}\rho$. From it we have $\langle \rho, f_k\rangle=0, \
k=0,1,\ldots,n, \ N=n+2.$ 
Therefore $\rho$ is a constant section  in
$\mathbb{C}^N$ and $\phi$ belongs to  case (iii).

If $\underline{g}_j \neq \underline{f}_i$. In
this case we claim $\langle
f_i, g_j\rangle =0$, i.e. $A''_{\phi}(\underline{ker}A^{'\perp}_{\phi})=0$ and then, $\phi$ belongs to  cases (i) or (ii) by Lemma 3.3.
For this purpose we put $\alpha=g_j-\frac{\langle g_j, f_i\rangle}{|f_i|^2}f_i$,
then from  \eqref{eq:4.2}  and expression  $\phi=\frac{\alpha\alpha^{*}}{|\alpha|^2}+\frac{f_if_i^{*}}{|f_i|^2}$,  direct computations show that
$$A_{\overline{z}} = \frac{1}{|f_i|^2}f_{i}f_{i+1}^{*}+\frac{1}{|g_{j-1}|^2}g_{j-1}g_{j}^{*}$$ 
$$
\begin{array}{lll} [A_{\overline{z}}, [A_z, A_{\overline{z}}]] & = &-\frac{2|f_{i+1}|^2}{|f_i|^4}f_{i}f_{i+1}^{*}-\frac{2|g_{j}|^2}{|g_{j-1}|^4}g_{j-1}g_{j}^{*}+\frac{|g_j|^2\langle g_{j-1}, f_{i+1}\rangle}{|f_i|^2|g_{j-1}|^4}f_{i}g_{j-1}^{*}+\frac{|f_{i+1}|^2\langle f_{i}, g_{j}\rangle}{|g_{j-1}|^2|f_{i}|^4}g_{j-1}f_{i}^{*}\\
& & +\frac{\langle g_{j-1}, f_{i+1}\rangle}{|f_i|^2|g_{j-1}|^2}f_{i+1}g_{j}^{*}+\frac{\langle f_{i}, g_{j}\rangle}{|g_{j-1}|^2|f_{i}|^2}g_{j}f_{i+1}^{*}.\end{array} $$
Using $[A_{\overline{z}}, [A_z, A_{\overline{z}}]]f_{i+1} = M_1A_{\overline{z}}f_{i+1}$ we find
$$[M_1\frac{|f_{i+1}|^2}{|f_i|^2}+2\frac{|f_{i+1}|^4}{|f_i|^4}-\frac{|g_j|^2\langle f_{i+1}, g_{j-1}\rangle\langle g_{j-1}, f_{i+1}\rangle}{|f_i|^2|g_{j-1}|^4}]f_i=\frac{|f_{i+1}|^2\langle f_i, g_j\rangle}{|f_i|^2|g_{j-1}|^2}g_j,$$
which gives
$$\langle f_i, g_j\rangle=0$$ from the supposition $\underline{f}_i\neq \underline{g}_j$.Thus the proof of our property is complete.
\end{proof}

From Proposition 4.1, to finish the complete classification of $\phi$ in \eqref{eq:4.2} with parallel second fundamental form, we distinguish three cases respectively:
 $\phi$ is a Frenet pair; $\phi$ is a mixed pair; whereas $\underline{\phi}=\underline{f}_i\oplus \underline{c}_0$.

\begin{lemma}
Let $\phi: S^2\rightarrow G(2,N;\mathbb{C})$ be a linearly full Frenet pair with parallel  second fundamental form  and  rank $\partial{'}\underline{\phi} = $ rank $\partial{''}\underline{\phi} = 1$, then up to $U(4)$ equivalence,
$\phi$ belongs to case (1) in Theorem 1.2.
\end{lemma}
\begin{proof}
Let us first assume that $\underline{\phi}=\underline{f}_{i-1}\oplus \underline{f}_i:S^2\rightarrow G(2,N;\mathbb{C})$, where $\underline{f}_i:S^2\rightarrow \mathbb{C}P^n$ is harmonic with $N=n+1$ and $\phi$ belongs to the following harmonic sequence
\begin{equation}0\stackrel{\partial{''}}
{\longleftarrow}\underline{f}_{0}\stackrel{\partial{''}}
{\longleftarrow}\cdots\stackrel{\partial{''}}
{\longleftarrow}\underline{f}_{i-2}\stackrel{\partial{''}}
{\longleftarrow}\underline{\phi}=\underline{f}_{i-1} \oplus \underline{f}_i\stackrel{\partial{'}} {\longrightarrow}
\underline{f}_{i+1}\stackrel{\partial{'}}
{\longrightarrow}\cdots\stackrel{\partial{'}}
{\longrightarrow}\underline{f}_{n}\stackrel{\partial{'}}
{\longrightarrow}0,\label{eq:4.3}\end{equation}
$i\geq2, \ n\geq i+1$.  On the one hand, by making use of  $\phi=\frac{f_{i-1}f_{i-1}^{*}}{|f_{i-1}|^2}+\frac{f_if_i^{*}}{|f_i|^2}$ and  $[A_{\overline{z}}, [A_z, A_{\overline{z}}]]=M_1A_{\overline{z}}$ we have
 \begin{equation}M_1=-2\frac{|f_{i+1}|^2}{|f_i|^2}=-2\frac{|f_{i-1}|^2}{|f_{i-2}|^2}=- \lambda^2,\label{eq:4.4}\end{equation}
which further implies $l_{i}=l_{i-2}$ and thus $\delta^{(n)}_{i}=\delta^{(n)}_{i-2}$.
Since the second fundamental form of $\phi$ is parallel, its Gauss curvature is a constant. By a similar reasoning as the one when $\phi$ is a Frenet pair in  subsection 3.1,  the harmonic sequence \eqref{eq:4.3} is totally unramified, which gives $\delta^{(n)}_{i}=(i+1)(n-i), \ \delta^{(n)}_{i-2}=(i-1)(n-i+2)$. So we get
\begin{equation}2i=n+1, \quad K+\frac{\|B\|^2}{2}=2. \label{eq:4.5}\end{equation}
On the other hand, by using $\phi$ it is suffices to prove that 
$$2P = \frac{|f_{i-2}|^2}{|f_{i-1}|^2|f_{i-3}|^2}f_{i-1}f_{i-3}^{*}+ \frac{|f_{i}|^2}{|f_{i-1}|^2|f_{i+1}|^2}f_{i+1}f_{i-1}^{*}
-\frac{1}{|f_{i-1}|^2}f_{i}f_{i-2}^{*}-\frac{1}{|f_{i+1}|^2}f_{i+2}f_{i}^{*}$$ and 
$$2[[A_{\overline{z}}, A_z], P]=\frac{1}{|f_{i-3}|^2}f_{i-1}f_{i-3}^{*}-\frac{1}{|f_{i}|^2}f_{i+2}f_{i}^{*}.$$
Then $[[A_{\overline{z}}, A_z], P]f_{i-1}=M_2Pf_{i-1}$  is equivalent to 
\begin{equation}M_2=0, \quad K=\frac{\|B\|^2}{4}. \label{eq:4.6}\end{equation}
Therefore combining \eqref{eq:4.5}  and \eqref{eq:4.6}   we have
\begin{equation} \quad K=\frac{2}{3}, \quad \|B\|^2=\frac{8}{3}.\label{eq:4.7}\end{equation}
It follows from \eqref{eq:4.5} and
$$\frac{2}{3}=K=-\frac{2}{\lambda^2}\partial\overline\partial\log \lambda^2=-\frac{1}{l_i}\partial\overline\partial\log l_i=2-\frac{l_{i-1}+l_{i+1}}{l_i}=2-\frac{\delta^{(n)}_{i-1}+\delta^{(n)}_{i+1}}{\delta^{(n)}_i}$$
that
$$ i=2,\quad n=3.$$

By  relation in \eqref{eq:4.4} and \eqref{eq:4.7} we find that $\underline{f}_2: S^2 \rightarrow \mathbb{C}P^3$ is of constant curvature, using the rigidity
theorem of Bolton et al (\cite{Bolton}), up to a holomorphic isometry of
$\mathbb{C}P^{3}$,  $\underline{f}_2$ is a Veronese surface. We
can choose a complex coordinate $z$ on
$\mathbb{C}=S^2\backslash\{pt\}$ so that $ {f}_2 = {U}{V}^{(3)}_2$,
where $U\in U(4)$ and ${V}^{(3)}_2$ has the standard expression
given in Section 2.

Finally, in proving Lemma 4.2, it is easy to check that, for any $U\in U(4)$, $$\underline{\phi}=\underline{UV}^{(3)}_1\oplus \underline{UV}^{(3)}_2:S^2\rightarrow G(2,4;\mathbb{C})$$ is of parallel second fundamental form with $ K=\frac{2}{3}$ and $\|B\|^2=\frac{8}{3}$, thus the proof of our lemma is complete.
\end{proof}

Lemma 4.2 proves the case that $\phi$ is a Frenet pair, more interesting is naturally the case that  $\phi$ is a mixed pair, which we are going to suppose from now on. Without loss of generality, we express it by $$\underline{\phi}=\underline{g}_m \oplus \underline{f}_0: S^2\rightarrow G(2,N;\mathbb{C}),$$ where $\underline{f}_0: S^2\rightarrow \mathbb{C}P^{n}$ is
holomorphic and $\underline{g}_m: S^2\rightarrow \mathbb{C}P^{m}$ is
anti-holomorphic for some $m< N$ and $n<N$. Then we  have the following harmonic sequence:
$$0\stackrel{\partial{''}}
{\longleftarrow}\underline{g}_{0}\stackrel{\partial{''}}
{\longleftarrow}\cdots\stackrel{\partial{''}}
{\longleftarrow}\underline{g}_{m-1}\stackrel{\partial{''}}
{\longleftarrow}\underline{\phi}=\underline{g}_m \oplus \underline{f}_0\stackrel{\partial{'}} {\longrightarrow}
\underline{f}_{1}\stackrel{\partial{'}}
{\longrightarrow}\cdots\stackrel{\partial{'}}
{\longrightarrow}\underline{f}_{n}\stackrel{\partial{'}}
{\longrightarrow}0.$$
Basing on the formula  $\phi=\frac{g_mg_m^{*}}{|g_m|^2}+\frac{f_0f_0^{*}}{|f_0|^2}$ we have
$$A_{\overline{z}} = \frac{1}{|f_0|^2}f_0f_1^{*}+\frac{1}{|g_{m-1}|^2}g_{m-1}g_m^{*}, \quad
\lambda^2=\frac{|f_1|^2}{|f_0|^2}+\frac{|g_m|^2}{|g_{m-1}|^2},$$

$$ [A_{\overline{z}}, [A_z, A_{\overline{z}}]]  =  -\frac{2|f_1|^2}{|f_0|^4}f_0f_1^{*}-\frac{2|g_m|^2}{|g_{m-1}|^4}g_{m-1}g_m^{*}
+\frac{|g_m|^2\langle g_{m-1}, f_1\rangle}{|f_0|^2|g_{m-1}|^4}f_0g_{m-1}^{*}+\frac{\langle g_{m-1}, f_1\rangle}{|f_0|^2|g_{m-1}|^2}f_1g_{m}^{*}.$$
Then analyzing  $[A_{\overline{z}}, [A_z, A_{\overline{z}}]]=M_1A_{\overline{z}}$ we find
\begin{equation}[M_1+2\frac{|f_1|^2}{|f_0|^2}]f_1=\frac{|g_m|^2\langle  f_1, g_{m-1}\rangle}{|g_{m-1}|^4}g_{m-1}, \quad [M_1+2\frac{|g_m|^2}{|g_{m-1}|^2}]g_{m-1}=\frac{\langle g_{m-1}, f_1\rangle}{|f_0|^2}f_1.\label{eq:4.8}\end{equation}

Now we prove

\begin{lemma}
Let $\phi: S^2\rightarrow G(2,N;\mathbb{C})$ be a linearly full mixed pair with parallel  second fundamental form  and  rank $\partial{'}\underline{\phi} = $ rank $\partial{''}\underline{\phi} = 1$, then
$\phi$ is congruent to cases (2) (3) (4) or (5) in Theorem 1.2.
\end{lemma}
\begin{proof}
From \eqref{eq:4.8}, to finish the classification of $\phi$, we distinguish two cases:
 $\underline{f}_1=\underline{g}_{m-1}$, whereas  $\underline{f}_1 \neq \underline{g}_{m-1}$.

\textbf{(a)}  $\underline{f}_1=\underline{g}_{m-1}$.
In this case we have
$$-M_1=\frac{|f_1|^2}{|f_0|^2}=\frac{|g_m|^2}{|g_{m-1}|^2}=\frac{\lambda^2}{2}$$
and $\underline{f}_{2}=\underline{g}_m$. Since $g_m$ is antiholomorphic, we have $\underline{f}_{2}$ is antiholomorphic, which establishes $$m=n=2.$$
where $\underline{f}_0, \ \underline{g}_{m}: S^2 \rightarrow \mathbb{C}P^2$ are both harmonic maps with  constant curvature. Therefore, by Lemma 2.2, there exists some $U\in U(3)$ s.t. $f_0=UV^{(2)}_0$, and the expression for $\phi$ becomes
$$ \underline{\phi}=\underline{f}_0\oplus \underline{f}_2=\underline{UV}_0^{(2)}\oplus \underline{UV}_2^{(2)}: S^2\rightarrow G(2,3;\mathbb{C}).$$
For any $U\in U(3)$, it can be easily calculated that such $\underline{\phi}=\underline{f}_0\oplus \underline{f}_2$ is totally geodesic with constant curvature $K=1$, and it is congruent to the case (2) in Theorem 1.2.

\textbf{(b)}  $\underline{f}_1 \neq \underline{g}_{m-1}$. In this case, it follows from relations shown in \eqref{eq:4.8} that $$\langle f_1, g_{m-1}\rangle=0, \ -M_1=2\frac{|f_1|^2}{|f_0|^2}=2\frac{|g_m|^2}{|g_{m-1}|^2}=\lambda^2,$$
which establishes \begin{equation}K+\frac{\|B\|^2}{2}=2, \quad m=n,\label{eq:4.9}\end{equation} and $\underline{f}_0, \ \underline{g}_{m}: S^2\rightarrow \mathbb{C}P^n$ are both harmonic maps  with constant curvature. Thus, by Lemma 2.2, they are both Veronese surfaces in $\mathbb{C}P^n$, up to a $U(n+1)$-motion.

If $m=n=1$. Here by a direct computation we immediately have $P=0$  and $\underline{\phi}=\underline{f}_0\oplus \underline{g}_1: S^2\rightarrow G(2,4;\mathbb{C})$ is totally geodesic with constant curvature $K=2$ under the assumption that $\{\underline{g}_0, \underline{g}_1\}$ and $\{\underline{f}_0, \underline{f}_1\}$ are mutually orthogonal in $\mathbb{C}P^1$ with constant curvature. Adding zeros to $V^{(1)}_0$ and $V^{(1)}_1$ respectively, let  $\widehat{V}^{(1)}_0=(0, \ 0,  \ 1, \ z)^T$ and $V^{(1)}_1=\frac{1}{1+z\overline{z}}(-\overline{z}, \ 1,  \ 0, \ 0)^T$,  then,  up to an
isometry of $G(2,4;\mathbb{C})$, there exists some $U\in U(4)$ s.t. $$\underline{\phi}=\underline{U}\underline{\widehat{V}}^{(1)}_0\oplus \underline{UV}^{(1)}_1: S^2\rightarrow G(2,4;\mathbb{C}),$$
 such $\phi$ is congruent to the case (3) in Theorem 1.2.

 If $m=n\geq 2$. We obtain
 $$2P = \frac{|g_{m-1}|^2}{|g_{m-2}|^2|g_{m}|^2}g_mg_{m-2}^{*}-\frac{1}{|f_1|^2}f_2f_0^{*}, $$
$$2[[A_{\overline{z}}, A_z], P]=\frac{1}{|g_{m-2}|^2}g_mg_{m-2}^{*}+\frac{|g_m|^2\langle f_2, g_{m-1}\rangle}{|f_1|^2|g_{m-1}|^4}g_{m-1}f_0^{*}-\frac{|g_{m-1}|^2\langle f_1, g_{m-2}\rangle}{|f_0|^2|g_{m-2}|^2|g_m|^2}g_{m}f_1^{*}-\frac{1}{|f_0|^2}f_2f_0^{*}.$$
Then $[[A_{\overline{z}}, A_z], P]=M_2P$ holds  if and only if the following equations
$$ (M_2-\frac{|f_1|^2}{|f_0|^2})f_2=-\frac{\langle f_2, g_{m-1}\rangle|g_m|^2}{|g_{m-1}|^4}g_{m-1}, $$
$$ (M_2-\frac{|g_m|^2}{|g_{m-1}|^2})g_{m-2}=-\frac{\langle  g_{m-2}, f_1 \rangle}{|f_0|^2}f_1$$
hold.

If $\underline{g}_{m-1}=\underline{f}_2$. In this case $\underline{g}_{m}=\underline{f}_3, \ m=n=3$ and $\underline{f}_0:S^2\rightarrow \mathbb{C}P^3$ is of constant curvature. Then, by Lemma 2.2, there exists some  $U\in U(4)$ s.t. $f_0=UV^{(3)}_0$ and
$$\underline{\phi}=\underline{f}_0\oplus \underline{f}_3=\underline{UV}^{(3)}_0\oplus \underline{UV}^{(3)}_3: S^2\rightarrow G(2,4;\mathbb{C}).$$  With a simple test we know that such $\phi$ is of parallel second fundamental form with $K=\frac{2}{3}$ and  $\|B\|^2=\frac{8}{3}$, and it is congruent to the case (4) in Theorem 1.2.

If $\underline{g}_{m-1} \neq \underline{f}_2$. In this case we have
$$\langle f_2, g_{m-1}\rangle=0, \quad M_2=\frac{|f_1|^2}{|f_0|^2}=\frac{|g_m|^2}{|g_{m-1}|^2}=\frac{\lambda^2}{2}.$$
Then combining it with \eqref{eq:4.9} we have $$K=1, \quad \|B\|^2=2.$$ 
So we have
$$1=K=2-\frac{\delta^{(n)}_{1}}{\delta^{(n)}_0}=2-\frac{\delta^{(m)}_{m-2}}{\delta^{(m)}_{m-1}}.$$
Thus
$$m=n=2,$$
and $\{\underline{g}_0, \underline{g}_1, \underline{g}_2\}$ and $\{\underline{f}_0, \underline{f}_1, \underline{f}_2\}$ are mutually orthogonal harmonic sequences in $\mathbb{C}P^2$ with constant curvature.
By a simple test we know $$\underline{\phi}=\underline{g}_2\oplus \underline{f}_0: S^2\rightarrow G(2,6;\mathbb{C})$$ is of parallel second fundamental form. Then let $\widehat{V}^{(2)}_0=(0, \ 0, \ 0, \ 1, \ \sqrt{2}z, \ z^2)^T$ and $V^{(2)}_2=\frac{2}{(1+z\overline{z})^2}(\overline{z}^2, \ -\sqrt{2}\overline{z}, \ 1, \ 0, \ 0, \ 0)^T$,  there exists some $U\in U(6)$ s.t.
 $$\underline{\phi}=\underline{U}\underline{\widehat{V}}^{(2)}_0\oplus \underline{UV}^{(2)}_2: S^2\rightarrow G(2,6;\mathbb{C}),$$
which is congruent to the case (5) in Theorem 1.2. Summing up, we get the conclusion.
\end{proof}

Lemma 4.3 gives a complete classification of $\phi$ with parallel second fundamental form  and  rank $\partial{'}\underline{\phi} = $ rank $\partial{''}\underline{\phi} = 1$ when it is a mixed pair. Finally we need to consider the case that  $\underline{\phi}=\underline{f}_i\oplus \underline{c}_0: S^2\rightarrow G(2,N;\mathbb{C})$, which is a linearly full harmonic map with  $\nabla B=0, \ 1\leq i \leq N-3$ and $c_0=(0,\ldots, 0,1)^T\in \mathbb{C}^N$, then we find the following harmonic sequence:
$$0\stackrel{\partial{''}}
{\longleftarrow}\underline{f}_{0}\stackrel{\partial{''}}
{\longleftarrow}\cdots\stackrel{\partial{''}}
{\longleftarrow}\underline{f}_{i-1}\stackrel{\partial{''}}
{\longleftarrow}\underline{\phi}=\underline{f}_{i} \oplus \underline{c}_0\stackrel{\partial{'}} {\longrightarrow}
\underline{f}_{i+1}\stackrel{\partial{'}}
{\longrightarrow}\cdots\stackrel{\partial{'}}
{\longrightarrow}\underline{f}_{n}\stackrel{\partial{'}}
{\longrightarrow}0,$$
$i\geq 1, \ n\geq i+1, \ N=n+2$. At first we use  $\phi=\frac{c_{0}c_{0}^{*}}{|c_{0}|^2}+\frac{f_if_i^{*}}{|f_i|^2}$ to compute
\begin{equation}   \lambda^2=\frac{|f_{i+1}|^2}{|f_i|^2}+\frac{|f_{i}|^2}{|f_{i-1}|^2}. \label{eq:4.10}\end{equation}
This implies further that $\underline{f}_i: S^2\rightarrow \mathbb{C}P^n$ is of constant curvature,
using the rigidity
theorem of Bolton et al (\cite{Bolton}), up to a holomorphic isometry of
$\mathbb{C}P^{n}$,  there exists some $U\in U(n+1)$ s.t. $ {f}_i = U{V}^{(n)}_i$.

\begin{lemma}
Let $\underline{\phi}=\underline{f}_i\oplus \underline{c}_0: S^2\rightarrow G(2,N;\mathbb{C})$ be a linearly full conformal minimal immersion with parallel  second fundamental form  and $1\leq i \leq N-3$,
then $\phi$ is congruent to cases (6) or (7)  in Theorem 1.2.
\end{lemma}
\begin{proof}
By using  relation $[A_{\overline{z}}, [A_z, A_{\overline{z}}]]=M_1A_{\overline{z}}$ we get
$-M_1=\frac{|f_{i+1}|^2}{|f_i|^2}=\frac{|f_{i}|^2}{|f_{i-1}|^2}$,
which establishes $l_{i-1}=l_i, \ \delta^{(n)}_{i-1}=\delta^{(n)}_i$
and therefore
 \begin{equation}2i=n, \quad K+\frac{\|B\|^2}{2}=1.\label{eq:4.11}\end{equation}

If $i=1$. Then we have $n=2$ and $$\underline{\phi}=\underline{UV}^{(2)}_1\oplus \underline{c}_0: S^2\rightarrow G(2,4;\mathbb{C})$$ for some $U\in U(4)$(adding zero to the end of $V^{(2)}_1$ s.t. it belongs  to $\mathbb{C}^4$).  With a simple test we know that such $\phi$ is totally geodesic  with $K=1$, and it is congruent to the case (6) in Theorem 1.2.

If $i\geq 2$. Here  we obtain
$$2P = \frac{|f_{i-1}|^2}{|f_{i}|^2|f_{i-2}|^2}f_{i}f_{i-2}^{*}-\frac{1}{|f_{i+1}|^2}f_{i+2}f_{i}^{*}, \quad   [[A_{\overline{z}}, A_z], P]=0,$$
which implies 
 \begin{equation}K=\frac{1}{3}, \quad \|B\|^2=\frac{4}{3}, \quad i=2  \label{eq:4.12}\end{equation} from $[[A_{\overline{z}}, A_z], P]=M_2P$ and  \eqref{eq:4.11}.
 With a simple test we know, for any $U\in U(6)$,  $$\underline{\phi}=\underline{UV}_{2}^{(4)} \oplus \underline{c}_0: S^2\rightarrow G(2,6;\mathbb{C})$$ is of parallel second fundamental form(adding zero to the end of $V^{(4)}_2$ s.t. it belongs  to $\mathbb{C}^6$), and it is congruent to the case (7) in Theorem 1.2. This finishes the proof.
\end{proof}

Summing Lemmas 4.2-4.4, we get Theorem 1.2 in Section 1.

\section{Minimal two-spheres with parallel second fundamental form and rank $\partial{'}\underline{\phi} = 1$,  rank $\partial{''}\underline{\phi} = 2$}

 In this section we analyze  conformal minimal immersions $\phi$ from $S^2$ to $G(2,N;\mathbb{C})$ with parallel second fundamental form, rank $\partial{'}\underline{\phi} = 1$ and rank $\partial{''}\underline{\phi} = 2$. From $\phi$, a harmonic sequence is derived as follows:
\begin{equation}0\stackrel{\partial{''}}
{\longleftarrow}\cdots\stackrel{\partial{''}}
{\longleftarrow}\underline{\phi}_{-1}\stackrel{\partial{''}}
{\longleftarrow}\underline{\phi}\stackrel{\partial{'}} {\longrightarrow}
\underline{\phi}_{1}\stackrel{\partial{'}}
{\longrightarrow}\cdots\stackrel{\partial{'}}
{\longrightarrow}0.\label{eq:5.1}\end{equation}
Since $\underline{\phi}_1$ is of rank one and it is harmonic, we can write $\underline{\phi}_1=\underline{f}_{i+1}$, where $f_{i+1}$ is a local section of $\underline{\phi}_1$ and it belongs to the following harmonic sequence in $\mathbb{C}P^n$ $$0\stackrel{\partial{'}}
{\longrightarrow}\underline{f}_0\stackrel{\partial{'}}
{\longrightarrow}\underline{f}_{1}\stackrel{\partial{'}}
{\longrightarrow}\cdots\stackrel{\partial{'}} {\longrightarrow}
\underline{f}_{i}\stackrel{\partial{'}}
{\longrightarrow}\cdots\stackrel{\partial{'}}
{\longrightarrow}\underline{f}_n\stackrel{\partial{'}}
{\longrightarrow}0$$ for some $1\leq i \leq n-1$, here $f_0, \ldots, f_n$ satisfy equations \eqref{eq:2.8} and \eqref{eq:2.9}.
From \eqref{eq:5.1}, since $f_i$ is a local section of $\underline{\phi}$  and rank $\underline{\phi}=2$, there exists another local section $\alpha$ of $\underline{\phi}$ such that $\underline{\phi}=\underline{\alpha}\oplus \underline{f}_i$. Set $$\alpha_1=\partial \alpha-\frac{\langle \partial \alpha, \alpha\rangle}{|\alpha|^2}\alpha, \ \alpha_{-1}=\overline{\partial} \alpha-\frac{\langle \overline{\partial} \alpha, \alpha\rangle}{|\alpha|^2}\alpha,   \quad  \beta=A^{''}_{\phi}(f_i), \ \underline{\gamma}=\underline{\beta}^{\bot}\cap \underline{\phi}_{-1},$$
then $\underline{\phi}_{-1}$ is spanned by local sections $\beta$ and $\gamma$.

 To characterize $\phi$ and give its classification, at first we state one of  Burstall and Wood'
results (\cite{Burstall}) as follows:
\begin{lemma}[{\cite{Burstall}}]
Let $\underline{\phi}:S^2\rightarrow G(2,N;\mathbb{C})$ be a  harmonic map with $\partial{'}\underline{\phi}$ of rank one and  $A''_{\phi}(\underline{ker}A^{'\perp}_{\phi})\neq 0$.  Let $\underline{\alpha}$ be the anti-holomorphic subbundle of $\underline{\phi}$ defined by $\underline{ker}A'_{\phi}$, then backward replacement of $\underline{\beta}=\underline{\alpha}^{\perp}\bigcap \underline{\phi}$ produces a new harmonic map $\underline{\widetilde{\phi}}=\underline{\alpha}\oplus \underline{Im}(A''_{\phi}|\underline{\beta}): S^2 \rightarrow G(2,N;\mathbb{C})$, where $\partial'\widetilde{\underline{\phi}}=\underline{\beta},  \ \partial^{(i)}\widetilde{\underline{\phi}}= \partial^{(i-1)}\underline{\phi}$ for $i\geq 2$.
\end{lemma}

Using this lemma, by backward replacement of $\underline{f}_i$, we obtain a new harmonic map $\underline{\phi}^{(1)}=\underline{\alpha}\oplus\underline{\beta}: S^2\rightarrow G(2,N;\mathbb{C})$, which belongs to the following harmonic sequence
\begin{equation}0\stackrel{\partial{''}}
{\longleftarrow}\cdots\stackrel{\partial{''}}
{\longleftarrow}\underline{\phi}^{(1)}=\underline{\alpha}\oplus \underline{\beta}\stackrel{\partial{'}} {\longrightarrow}
\underline{f}_{i}\stackrel{\partial{'}} {\longrightarrow}
\cdots\stackrel{\partial{'}} {\longrightarrow}
\underline{f}_{n}\stackrel{\partial{'}}
{\longrightarrow}0.\label{5.2}\end{equation}
 Then $f_{i-1}$ is a local section of $\underline{\alpha}\oplus \underline{\beta}$ and
$$\alpha_1=\frac{\langle \alpha, f_{i-1}\rangle}{|f_{i-1}|^2}f_i, \ \alpha_{-1}=\frac{\langle \overline{\partial} \alpha, \beta\rangle}{|\beta|^2}\beta+\frac{\langle \overline{\partial} \alpha, \gamma\rangle}{|\gamma|^2}\gamma.$$
By making use of  $\phi=\frac{\alpha\alpha^{*}}{|\alpha|^2}+\frac{f_if_i^{*}}{|f_i|^2}$,  it was proved that
\begin{equation}A_{\overline{z}} = \frac{f_if_{i+1}^{*}}{|f_i|^2}+ \frac{\langle f_{i-1}, \beta\rangle}{|f_{i-1}|^2|\beta|^2}\beta f_i^{*} -
\frac{\alpha_{-1}\alpha^{*}}{|\alpha|^2},    \quad    \lambda^2=\frac{|f_{i+1}|^2}{|f_i|^2}+\frac{|\langle \beta, f_{i-1}\rangle|^2|f_i|^2}{|f_{i-1}|^4|\beta|^2}+\frac{|\alpha_{-1}|^2}{|\alpha|^2},\label{eq:5.2}\end{equation}

$$
\begin{array}{lll} [A_{\overline{z}}, [A_z, A_{\overline{z}}]] & = &
[\frac{|\langle f_{i-1}, \beta\rangle|^2}{|f_{i-1}|^4|\beta|^2}-2\frac{|f_{i+1}|^2}{|f_i|^4}]f_if_{i+1}^{*}
+[\frac{\langle f_{i-1}, \beta\rangle |f_{i+1}|^2}{|f_{i-1}|^2|f_{i}|^2|\beta|^2}-2\frac{|\langle f_{i-1}, \beta\rangle|^2\langle f_{i-1}, \beta\rangle|f_i|^2}{|f_{i-1}|^6|\beta|^4}]\beta f_i^{*}
\\
& & +2\frac{|\alpha_{-1}|^2}{|\alpha|^4}\alpha_{-1}\alpha^{*}+2\frac{|\langle f_{i-1}, \beta\rangle|^2\langle \alpha_{-1}, \beta\rangle|f_i|^2}{|f_{i-1}|^4|\alpha|^2|\beta|^4}\beta\alpha^{*}
-2\frac{\langle f_{i-1}, \beta\rangle\langle \beta, \alpha_{-1}\rangle}{|f_{i-1}|^2|\alpha|^2|\beta|^2}\alpha_{-1}f_i^{*}
\\
& & +\frac{|\langle \beta, f_{i-1}\rangle|^2\langle \beta, f_{i+1}\rangle}{|f_{i-1}|^4|\beta|^4}f_i\beta^{*}
+\frac{\langle \alpha_{-1}, f_{i+1}\rangle}{|\alpha|^2|f_i|^2}f_i\alpha_{-1}^{*}
+\frac{\langle \beta, f_{i+1}\rangle\langle  f_{i-1}, \beta\rangle}{|f_{i-1}|^2|f_i|^2|\beta|^2}f_{i+1}f_i^{*}
\\
& &
-\frac{\langle \alpha_{-1}, f_{i+1}\rangle}{|\alpha|^2|f_i|^2}f_{i+1}\alpha^{*}
-\frac{\langle \beta, \alpha_{-1}\rangle\langle f_{i-1}, \beta\rangle}{|f_{i-1}|^2|\alpha|^2|\beta|^2}\alpha f_{i+1}^{*}.\end{array} $$
Then we have
\begin{equation}\langle \alpha_{-1}, \beta\rangle =0, \ \langle \beta, f_{i+1}\rangle =0,\label{5.3}\end{equation}
\begin{equation}  M_1+2\frac{|\langle f_{i-1}, \beta\rangle|^2|f_i|^2}{|f_{i-1}|^4|\beta|^2}=\frac{|f_{i+1}|^2}{|f_i|^2},\label{eq:5.3}\end{equation}
\begin{equation}(M_1+2\frac{|\alpha_{-1}|^2}{|\alpha|^2})\alpha_{-1}=\frac{\langle \alpha_{-1}, f_{i+1}\rangle}{|f_i|^2}f_{i+1},\label{eq:5.4}\end{equation}
\begin{equation}[M_1-\frac{|\langle f_{i-1}, \beta\rangle|^2|f_i|^2}{|f_{i-1}|^4|\beta|^2}+2\frac{|f_{i+1}|^2}{|f_i|^2}]f_{i+1}=\frac{\langle f_{i+1}, \alpha_{-1}\rangle}{|\alpha|^2}\alpha_{-1} \label{eq:5.5}\end{equation}
from the fact $[A_{\overline{z}}, [A_z, A_{\overline{z}}]]=M_1A_{\overline{z}}$ because $\nabla B=0$.
It will be convenient in the following for us to put
\begin{equation}\frac{|\langle f_{i-1}, \beta\rangle|^2|f_i|^2}{|f_{i-1}|^4|\beta|^2}=a\frac{|f_{i+1}|^2}{|f_i|^2}, \
\frac{|\alpha_{-1}|^2}{|\alpha|^2}=b\frac{|f_{i+1}|^2}{|f_i|^2},  $$  and  $$\lambda_1=-\frac{1}{\lambda^2|f_{i}|^2}, \ \lambda_2=\frac{1}{\lambda^2|\alpha|^2}, \ \lambda_3=-\frac{\langle \beta, f_{i-1}\rangle}{\lambda^2|f_{i-1}|^2|\beta|^2}. \label{eq:5.6}\end{equation}
By analysis  \eqref{eq:5.3}-\eqref{eq:5.5} we can derive that $a, \ b$ are both constants. 
Then
 applying  equations $P=\partial(\frac{A_z}{\lambda^2})$ and $\partial \lambda_1+\lambda_1\frac{\langle \partial f_{i+1}, f_{i+1}\rangle}{|f_{i+1}|^2}=0$ we obtain
\begin{equation}
\begin{array}{lll} P & = & \lambda_1f_{i+2}f_i^{*}-\lambda_1\frac{|f_i|^2}{|f_{i-1}|^2}f_{i+1}f_{i-1}^{*}+ \partial \lambda_2 \alpha\alpha^{*}_{-1}+ \lambda_2\frac{\langle \partial \alpha, \alpha\rangle}{|\alpha|^2}\alpha\alpha_{-1}^{*}+\lambda_2\frac{\langle \alpha, f_{i-1}\rangle}{|f_{i-1}|^2}f_i\alpha_{-1}^{*}\\
& &+\lambda_2\alpha(\overline{\partial}\alpha_{-1})^{*}+\partial \lambda_3 f_i \beta^{*}+\lambda_3 f_{i+1}\beta^{*}  + \lambda_3\frac{\langle \partial f_i, f_i\rangle}{|f_i|^2}f_i \beta^{*} + \lambda_3 f_i (\overline{\partial}\beta)^{*}.\end{array} \label{eq:5.7}\end{equation}
 With it relation $[[A_{\overline{z}}, A_z], P]=M_2P$ is equivalent to the following  four equations
\begin{equation}
\begin{array}{lll} \overline{\lambda}_2(M_2-\frac{|\alpha_{-1}|^2}{|\alpha|^2})\overline{\partial} \alpha_{-1} & = &
[(2\frac{|\alpha_{-1}|^2}{|\alpha|^2}-M_2)(\overline{\partial\lambda}_{2}+\overline{\lambda}_{2}\frac{\langle \alpha, \partial \alpha\rangle}{|\alpha|^2})+\overline{\lambda}_2\frac{\langle \overline{\partial} \alpha_{-1}, \alpha_{-1}\rangle}{|\alpha|^2}]\alpha_{-1}\\
& &+[(1-a)\overline{\lambda}_1\frac{|f_{i+1}|^2\langle \alpha, f_{i+2}\rangle}{|f_i|^2|\alpha|^2}+(1-a)\overline{\lambda}_2\frac{|f_{i+1}|^2\langle \overline{\partial}\alpha_{-1}, f_{i}\rangle}{|f_i|^4}\\
& &+\overline{\lambda}_1\frac{|\alpha_{-1}|^2\langle \alpha, f_{i+2}\rangle}{|\alpha|^4}-M_2\overline{\lambda}_1\frac{\langle \alpha, f_{i+2}\rangle}{|\alpha|^2}]f_{i}\\
& &-[\overline{\partial\lambda}_2\frac{\langle \alpha_{-1}, f_{i+1}\rangle}{|f_i|^2}+\overline{\lambda}_2\frac{\langle \alpha_{-1}, f_{i+1}\rangle\langle \alpha, \partial \alpha\rangle}{|\alpha|^2|f_i|^2}+\overline{\lambda}_2\frac{\langle \overline{\partial} \alpha_{-1}, f_{i+1}\rangle}{|f_i|^2}]f_{i+1},\end{array} \label{eq:5.8}\end{equation}

\begin{equation}
\begin{array}{lll} \overline{\lambda}_3[M_2-(a-1)\frac{|f_{i+1}|^2}{|f_i|^2}]\overline{\partial} \beta & = &[(2a-1)\frac{|f_{i+1}|^2}{|f_i|^2}(\overline{\partial\lambda}_{3}+\overline{\lambda}_3\frac{\langle f_i, \partial f_i\rangle}{|f_i|^2})+a\overline{\lambda}_3\frac{|f_{i+1}|^2\langle \overline{\partial} \beta, \beta\rangle}{|f_i|^2|\beta|^2}-M_2\overline{\partial\lambda}_{3}\\
& &-M_2\overline{\lambda}_3\frac{\langle f_i, \partial f_i\rangle}{|f_i|^2}]\beta +[(a+b-1)\overline{\lambda}_2\frac{|f_{i+1}|^2\langle f_{i-1}, \alpha\rangle}{|f_i|^2|f_{i-1}|^2}+\overline{\lambda}_3\frac{\langle \overline{\partial} \beta, \alpha_{-1}\rangle}{|\alpha|^2}\\
& &-\overline{\lambda}_2M_2\frac{\langle f_{i-1}, \alpha\rangle}{|f_{i-1}|^2}]\alpha_{-1}-[\overline{\lambda}_2\frac{\langle \alpha_{-1}, f_{i+1}\rangle\langle f_{i-1}, \alpha\rangle}{|f_i|^2|f_{i-1}|^2}
+\overline{\lambda}_3\frac{\langle \overline{\partial} \beta, f_{i+1}\rangle}{|f_i|^2}]f_{i+1},\end{array} \label{eq:5.9}\end{equation}

\begin{equation}
\begin{array}{lll} M_2(\lambda_3f_{i+1}\beta^{*}-\lambda_1\frac{|f_{i}|^2}{|f_{i-1}|^2}f_{i+1}f_{i-1}^{*}) & = &(a+1)\lambda_3\frac{|f_{i+1}|^2}{|f_{i}|^2}f_{i+1}\beta^{*}-\lambda_1\frac{|f_{i+1}|^2}{|f_{i-1}|^2}f_{i+1}f_{i-1}^{*}\\
& &+\lambda_1\frac{|f_{i}|^2\langle f_{i+1}, \alpha_{-1}\rangle}{|\alpha|^2|f_{i-1}|^2}\alpha_{-1}f_{i-1}^{*}
-\lambda_3\frac{\langle f_{i+1}, \alpha_{-1}\rangle}{|\alpha|^2}\alpha_{-1}\beta^{*}\\
& &+\lambda_1\frac{|f_{i}|^2|\alpha_{-1}|^2\langle \alpha, f_{i-1}\rangle}{|f_{i-1}|^2|\alpha|^4}f_{i+1}\alpha^{*}-a
\lambda_1\frac{|f_{i+1}|^2\langle \beta, f_{i-1}\rangle}{|f_{i-1}|^2|\beta|^2}f_{i+1}\beta^{*},\end{array} \label{eq:5.10}\end{equation}

\begin{equation}[M_2-(1-a)\frac{|f_{i+1}|^2}{|f_i|^2}](f_{i+2}-\frac{\langle f_{i+2}, \alpha\rangle}{|\alpha|^2}\alpha)=-\frac{\langle f_{i+2}, \alpha_{-1}\rangle}{|\alpha|^2}\alpha_{-1}-a
\frac{|f_{i+1}|^2\langle f_{i+2}, \beta\rangle}{|\beta|^2|f_i|^2}\beta.\label{eq:5.11}\end{equation}
 From  \eqref{eq:5.4} and \eqref{eq:5.5}, in order to get the explicit expression of $\phi$, we distinguish two cases:
$\underline{\alpha}_{-1}=\underline{f}_{i+1}$, whereas
$\underline{\alpha}_{-1}\neq \underline{f}_{i+1}$.

\subsection{The case $\underline{\alpha}_{-1}=\underline{f}_{i+1}$.}

We observe that in this case $a=b=\frac{3}{4}$, and by \eqref{eq:5.3}-\eqref{eq:5.5} we immediately have  the following  formulae
\begin{equation}\lambda^2=\frac{5|f_{i+1}|^2}{2|f_i|^2}, \ M_1=-\frac{|f_{i+1}|^2}{2|f_i|^2},\label{eq:5.12}\end{equation}
which shows $$K+\frac{\|B\|^2}{2}=\frac{2}{5}.$$
Relation \eqref{eq:5.10} then gives
\begin{equation}(M_2+\frac{|f_{i+1}|^2}{2|f_i|^2})(\overline{\lambda}_1\frac{|f_{i+1}|^2}{|f_{i-1}|^2}f_{i-1}-\overline{\lambda}_3\frac{|f_{i+1}|^2}{|f_{i}|^2}\beta)=0.\label{eq:5.13}\end{equation}
Of importance is the induced metric of $\phi$ shown in \eqref{eq:5.12}, which can be interpreted that
 $\underline{f}_i: S^2\rightarrow \mathbb{C}P^n$ is harmonic with constant curvature.
Then using the rigidity
theorem of Bolton et al (\cite{Bolton}), up to a holomorphic isometry of
$\mathbb{C}P^{n}$,  there exists some $U\in U(n+1)$ s.t. $ {f}_i = U{V}^{(n)}_i$. We therefore establish the following lemma:

\begin{lemma}
Let $\phi: S^2\rightarrow G(2,N;\mathbb{C})$ be a linearly full conformal minimal immersion in \eqref{eq:5.1} with  $\nabla B=0$ and $\underline{\alpha}_{-1}=\underline{f}_{i+1}$, then
up to $U(4)$ equivalence, $\phi$ belongs to case (1) in Theorem 1.3.
\end{lemma}
\begin{proof}
Here    we first claim that $$M_2+\frac{|f_{i+1}|^2}{2|f_i|^2}\neq 0.$$ Otherwise if $M_2+\frac{|f_{i+1}|^2}{2|f_i|^2}= 0$, it means that 
    $K=0, \quad \|B\|^2=\frac{4}{5}$,  which is impossible. Then by \eqref{eq:5.13}, we use  the equation $\overline{\lambda}_1\frac{|f_{i+1}|^2}{|f_{i-1}|^2}f_{i-1}=\overline{\lambda}_3\frac{|f_{i+1}|^2}{|f_{i}|^2}\beta$ to prove $$\underline{\beta}=\underline{f}_{i-1}, \quad \underline{\phi}_{-1}=\underline{f}_{i-1}\oplus \underline{f}_{i+1}, \quad n=i+2,$$
which implies that
    $$i=1, \ n=3, \ \underline{\phi}=\underline{f}_{1}\oplus \underline{f}_{3}: S^2\rightarrow G(2,4;\mathbb{C}).$$
Finally it is easy to check that, for any $U\in U(4)$,  $$\underline{\phi}=\underline{UV}_{1}^{(3)} \oplus \underline{UV}_{3}^{(3)}: S^2\rightarrow G(2,4;\mathbb{C})$$ is totally geodesic with $K=\frac{2}{5}$. So we get the conclusion.
\end{proof}

\subsection{The case $\underline{\alpha}_{-1}\neq \underline{f}_{i+1}$.}

It is importance to rewritten the formulae in \eqref{eq:5.3}-\eqref{eq:5.5}  as
\begin{equation}\langle \alpha_{-1}, f_{i+1}\rangle =0,  \quad  \frac{\langle \beta, f_{i-1}\rangle\langle f_{i-1}, \beta\rangle|f_i|^2}{|f_{i-1}|^4|\beta|^2}=\frac{|f_{i+1}|^2}{|f_i|^2}=\frac{2|\alpha_{-1}|^2}{|\alpha|^2}, \ a=1, \ b=\frac 1 2,
\label{eq:5.14}\end{equation}
\begin{equation}\lambda^2=\frac{5|f_{i+1}|^2}{2|f_i|^2}, \ M_1=-\frac{|f_{i+1}|^2}{|f_i|^2},\label{eq:5.15}\end{equation}
which shows \begin{equation}K+\frac{\|B\|^2}{2}=\frac{4}{5}.\label{eq:5.16}\end{equation}From the metric given in \eqref{eq:5.15} and applying $\nabla B=0$,   up to a holomorphic isometry of
$\mathbb{C}P^{n}$,  there exists some $U\in U(n+1)$ s.t. $ {f}_i = U{V}^{(n)}_i$.
Especially,   equation \eqref{eq:5.10} can be transformed as
\begin{equation}(M_2-\frac{|f_{i+1}|^2}{2|f_i|^2})\langle f_{i-1}, \alpha\rangle=0.\label{5.19}\end{equation}

In harmonic sequence \eqref{5.2}, $f_{i-1}$ is a local section of $\underline\alpha\oplus\underline\beta$. Then 
$\langle f_{i-1}, \alpha\rangle=0$ is equivalent to $\underline{\beta}=\underline{f}_{i-1}$.
With it, to classify $\phi$, we shall divide our discussion into two cases, according as
 $\underline{\beta}=\underline{f}_{i-1}$, whereas $\underline{\beta}\neq \underline{f}_{i-1}$.
In the following we discuss these two cases respectively to prove the following two lemmas.

\begin{lemma}
Let $\phi:S^2\rightarrow G(2,N;\mathbb{C})$ be a linearly full conformal minimal immersion in \eqref{eq:5.1} with  $\nabla B=0$ and $\underline{\alpha}_{-1}\neq \underline{f}_{i+1}, \ \underline{\beta}=\underline{f}_{i-1}$, then up to an isometry of
$G(2,N;\mathbb{C})$, $\phi$ belongs to  case (2) in Theorem 1.3.
\end{lemma}
\begin{proof}
$\underline{\beta}=\underline{f}_{i-1}$ preserves the following relation
$$\frac{\langle \beta, f_{i-1}\rangle\langle f_{i-1}, \beta\rangle|f_i|^2}{|f_{i-1}|^4|\beta|^2}=\frac{|f_{i}|^2}{|f_{i-1}|^2}, \quad \langle f_{i+2}, \beta\rangle=0.$$
This together with \eqref{eq:5.14} show us
$l_{i-1}=l_i,\quad \delta^{(n)}_{i-1}=\delta^{(n)}_i$.
So we get \begin{equation}n=2i, \quad a=1.\label{eq:5.17}\end{equation}
In the following we discuss $\phi$ in cases $n\geq i+2, \ M_2 =0$;  $n\geq i+2, \ M_2 \neq0$ and
$n=i+1$ respectively.

\emph{At first}, \emph{if  $n\geq i+2, \ M_2 =0$.} By comparing $K=\frac{\|B\|^2}{4}$ and \eqref{eq:5.16}   $$K=\frac{4}{15}, \quad \|B\|^2=\frac{16}{15}.$$ In substituting the metric of $\phi$ shown in \eqref{eq:5.15} into the equation $K=\frac{4}{15}$, we have
$$\frac{4}{15}=K=-\frac{2}{\lambda^2}\partial\overline\partial\log \lambda^2=-\frac{4}{5l_i}\partial\overline\partial\log l_i=\frac 4 5-\frac{4l_{i+1}}{5l_i}=\frac 4 5-\frac{4\delta^{(n)}_{i+1}}{5\delta^{(n)}_i},$$
which implies
$$i=2, \quad n=4.$$
With a little change of notation,  equations \eqref{eq:5.8}-\eqref{eq:5.11} can be rewritten  in the form
\begin{equation}\partial\lambda_2+\lambda_2\frac{\langle  \partial \alpha, \alpha\rangle}{|\alpha|^2} + \lambda_2\frac{\langle \alpha_{-1}, \overline{\partial} \alpha_{-1} \rangle}{|\alpha_{-1}|^2}=0,\label{eq:5.18}\end{equation}
\begin{equation}\partial\lambda_3+\lambda_3\frac{\langle  \partial f_i, f_i\rangle}{|f_i|^2} + \lambda_3\frac{\langle \beta, \overline{\partial} \beta \rangle}{|\beta|^2}=0, \label{eq:5.19}\end{equation}
$$\langle \alpha_{-1}, f_{i+2}\rangle=0, \ \langle \beta, f_{i+2}\rangle=0, \ \langle \overline{\partial}\beta, \alpha_{-1}\rangle=0, \ \overline{\partial}\alpha_{-1}=\frac{\langle \overline{\partial} \alpha_{-1}, \alpha_{-1} \rangle}{|\alpha_{-1}|^2}\alpha_{-1}.$$Then expression of $P$  becomes
 $$P=\lambda_1f_4f_2^{*}+\lambda_3\frac{\langle f_0, \overline{\partial} \beta\rangle}{|f_0|^2}f_2f_0^{*}$$ 
 by  substituting \eqref{eq:5.18} \eqref{eq:5.19} into \eqref{eq:5.7}.  Clearly, straightforward  calculations give  the square of the length of the second fundamental form
$$\|B\|^2= 4tr PP^*=4[|\lambda_1|^2|f_2|^2|f_4|^2+|\lambda_3|^2\frac{\langle f_0, \overline{\partial} \beta\rangle\langle \overline{\partial} \beta, f_0\rangle|f_2|^2}{|f_0|^2}]=\frac{64}{75},$$which contradicts the fact that $\|B\|^2=\frac{16}{15}$.

\emph{Next}, \emph{if $n\geq i+2, \ M_2 \neq0$.} We shall prove that this assumption is also not true. 
For this purpose, by using of the first relation in \eqref{eq:5.14} we get
$$\langle f_{i+2}, \alpha\rangle=\langle \partial f_{i+1}, \alpha\rangle=-\langle f_{i+1}, \overline\partial \alpha\rangle=-\langle f_{i+1}, \alpha_{-1}\rangle=0.$$
This together with $\langle f_{i+2}, \beta\rangle=0$ and $a=1$ reduce 
 \eqref{eq:5.11} to
$M_2f_{i+2}=-\frac{\langle f_{i+2}, \alpha_{-1}\rangle}{|\alpha|^2}\alpha_{-1}$,  it further  implies that \begin{equation}\underline{\alpha}_{-1}=\underline{f}_{i+2}, \quad M_2=-\frac{|f_{i+1}|^2}{2|f_i|^2},\label{eq:5.20}\end{equation} then we find
$$K=\frac{2}{15}, \quad \|B\|^2=\frac{4}{3}$$
by  \eqref{eq:5.15} \eqref{eq:5.16} and \eqref{eq:5.20}. Observe the induced metric and Gauss curvature of $\phi$, similarly we verify 
$i=3, \quad n=6$ and $$\underline{\phi}=\underline{f}_{3}\oplus \underline{f}_{6}:S^2\rightarrow G(2,7;\mathbb{C}).$$
Applying Lemma 3.2,  with a straightforward calculation we know that, for any $U\in U(7)$, $$\underline{\phi}=\underline{UV}_{3}^{(6)} \oplus \underline{UV}_{6}^{(6)}: S^2\rightarrow G(2,7;\mathbb{C})$$ does not have parallel second fundamental form.

\emph{At last}, \emph{if $n=i+1$.} This relation together with   \eqref{eq:5.17} implies $i=1, \quad n=2$.
From  \eqref{eq:2.6} and \eqref{eq:5.15}, the Gauss curvature is
$$K=-\frac{2}{\lambda^2}\partial\overline\partial \log\lambda^2=-\frac{4}{5l_1}\partial\overline\partial \log l_1=
\frac 8 5 - \frac{4l_0}{5l_1}=\frac 8 5 - \frac{4\delta^{(2)}_0}{5\delta^{(2)}_1}=\frac 4 5.
$$
Then $$\|B\|^2=0, \quad M_2=\frac{2|f_{i+1}|^2}{|f_i|^2}.$$
Substituting it into \eqref{eq:5.8}, we get $$\overline{\partial}\alpha_{-1}=\frac{\langle \overline{\partial} \alpha_{-1}, \alpha_{-1} \rangle}{|\alpha_{-1}|^2}\alpha_{-1}.$$
Then by using \eqref{eq:2.1}, the subbundle  $\underline{\alpha}$ is harmonic in $\mathbb{C}P^1$ since $\underline{\phi}=\underline{\alpha}\oplus \underline{f}_i$ is harmonic and $\underline{\beta}=\underline{f}_{i-1}$.
Hence there exists  local sections $g_0, \ g_1$  such that  $$\underline{\alpha}_{-1}=\underline{g}_0, \ \underline{\alpha}=\underline{g}_1,$$ where $g_0$ is holomorphic (without loss of generality, we assume $\overline{\partial}g_0=0$) and $\{\underline{g}_0, \underline{g}_1\}$ and $\{\underline{f}_0, \underline{f}_1, \underline{f}_2\}$ are mutually orthogonal harmonic sequences in $\mathbb{C}P^1$ and $\mathbb{C}P^2$ respectively with constant curvature. Adding zeros to $V^{(1)}_1$ and $V^{(2)}_1$ respectively, let $\widehat{V}^{(1)}_1=\frac{1}{1+z\overline{z}}(0, \ 0, \ 0, \ -\overline{z}, \ 1)^T$ and $V^{(2)}_1=\frac{1}{1+z\overline{z}}(-2\overline{z}, \ \sqrt{2}(1-z\overline{z}), \ 2z, \ 0, \ 0)^T$, then,  up to an isometry of $G(2,5;\mathbb{C})$, there exists some $U\in U(5)$,
 $$\underline{\phi}
=\underline{U}\underline{\widehat{V}}_{1}^{(1)} \oplus \underline{UV}_{1}^{(2)}: S^2\rightarrow G(2,5;\mathbb{C}).$$ By an immediately computation, such $\phi$ is totally geodesic with $K=\frac{4}{5}$, it belongs to  case (2) in Theorem 1.3. Thus we get the conclusion.
\end{proof}

In the following we discuss the  case $\underline{\beta}\neq \underline{f}_{i-1}$. Since $f_{i-1}$ is a local section of $\underline\alpha\oplus\underline\beta$, $\underline{\beta}\neq \underline{f}_{i-1}$ means $\langle  f_{i-1}, \alpha\rangle\neq 0$. From \eqref{5.19} we have $M_2=\frac{|f_{i+1}|^2}{2|f_{i}|^2}$. Then there holds $K-\frac{\|B\|^2}{4}=\frac{1}{5}$. By combining it with \eqref{eq:5.16}  we conclude
$$K=\frac{2}{5}, \quad \|B\|^2=\frac{4}{5}.$$ 
Now by using of $M_2=\frac{|f_{i+1}|^2}{2|f_{i}|^2}$ and the first three relations in \eqref{eq:5.14},
\eqref{eq:5.9} can be reduced to
$$f_{i+2}=-\frac{\langle f_{i+2}, \alpha_{-1}\rangle}{|\alpha|^2}\alpha_{-1}-2\frac{\langle f_{i+2}, \beta\rangle}{|\beta|^2}\beta.$$
This together with the first relation in \eqref{5.3} show $f_{i+2}=0$, i.e., $n=i+1$. Moreover, it follows from $K=\frac{2}{5}$ that 
$$i=3, \quad n=4.$$
Then 
from \eqref{eq:5.8}-\eqref{eq:5.11} we find
$\overline{\partial}\beta =\frac{\langle \overline{\partial} \beta, \alpha_{-1} \rangle}{|\alpha_{-1}|^2}\alpha_{-1} + \frac{\langle \overline{\partial} \beta, \beta \rangle}{|\beta|^2}\beta$, and
relations (5.18) and (5.19) also hold here.
In order to give a  complete  classification of such $\phi$, we need the following important equation
 $$\langle \beta, f_{1}\rangle=0,$$ which is obtained  by substituting the expression of $\lambda_3$  into \eqref{eq:5.19}, and using \eqref{eq:5.15}. Finally, we end our classification for such harmonic maps  by the following lemma.

\begin{lemma}
Let $\phi: S^2\rightarrow G(2,N;\mathbb{C})$  be a  linearly full conformal minimal immersion  in \eqref{eq:5.1} with  $\nabla B=0$ and $\underline{\alpha}_{-1}\neq \underline{f}_{i+1}, \ \underline{\beta}\neq \underline{f}_{i-1}$, then up to an isometry of $G(2,N;\mathbb{C})$,  $\phi$ belongs to case (3) in Theorem 1.3.
\end{lemma}
\begin{proof}
By previous analysis, $\phi^{(1)}$ gives rise  to the following harmonic sequence
\begin{equation}0\stackrel{\partial{''}}
{\longleftarrow}\cdots\stackrel{\partial{''}}
{\longleftarrow}\underline{\alpha}_{-1}\stackrel{\partial{''}}
{\longleftarrow}\underline{\phi}^{(1)}=\underline{\alpha}\oplus \underline{\beta}\stackrel{\partial{'}} {\longrightarrow}
\underline{f}_{3}\stackrel{\partial{'}} {\longrightarrow}
\underline{f}_{4}\stackrel{\partial{'}}
{\longrightarrow}0.\label{eq:5.23}\end{equation}

At first we claim
\begin{equation}A^{''}_{\phi^{(1)}}(\underline{ker}A^{'\bot}_{\phi^{(1)}})\neq0.\label{eq:5.24}\end{equation}
Otherwise if $A^{''}_{\phi^{(1)}}(\underline{ker}A^{'\bot}_{\phi^{(1)}})=0$, then according to Lemma 3.3 we conclude $\phi^{(1)}$ is a Frenet pair and $\underline{\phi}^{(1)}=\underline{f}_1\oplus \underline{f}_2$. Hence we have $\underline{\beta}=\underline{f}_{2}$ by means of the equation  $\langle \beta, f_{1}\rangle=0$ , which contradicts the fact $\underline{\beta}\neq\underline{f}_{i-1}$.
Thus \eqref{eq:5.24}  holds.

In view of  properties of harmonic sequence \eqref{eq:5.23} we notice that,  $\underline{f}_{2}$ is a subbundle with rank one of $\underline{\phi}^{(1)}$, let $$\underline{\gamma} = \underline{f}^{\bot}_{2} \cap \underline{\phi}^{(1)},$$ then $\underline{\phi}^{(1)}$ can be rewritten as
$\underline{\phi}^{(1)}=\underline{\gamma}\oplus\underline{f}_2:S^2\rightarrow G(2,N;\mathbb{C})$.
Here $\underline{\gamma}$ is an anti-holomorphic subbundle of $\underline{\phi}^{(1)}$, it satisfies $A'_{\phi^{(1)}}(\gamma)=0$ and $A''_{\phi^{(1)}}(f_{2})\neq 0$, i.e. $\underline{\gamma}=\underline{ker}A'_{\phi^{(1)}}, \ A''_{\phi^{(1)}}(\underline{ker}A^{'\perp}_{\phi^{(1)}})\neq 0$. Then Lemma 5.1 shows that, the backward replacement of $\underline{f}_{2}$ produces a new harmonic map $$\underline{\phi}^{(2)}=\underline{\alpha}_{-1}\oplus \underline{\gamma}: S^2\rightarrow G(2,N;\mathbb{C}),$$ it derives a harmonic sequence as follows:
$$0\stackrel{\partial{''}}
{\longleftarrow}\cdots\stackrel{\partial{''}}
{\longleftarrow}\underline{\alpha}_{-2}\stackrel{\partial{''}}
{\longleftarrow}\underline{\phi}^{(2)}=\underline{\alpha}_{-1}\oplus \underline{\gamma}\stackrel{\partial{'}} {\longrightarrow}
\underline{f}_{2}\stackrel{\partial{'}} {\longrightarrow}
\underline{f}_{3}\stackrel{\partial{'}} {\longrightarrow}
\underline{f}_{4}\stackrel{\partial{'}}
{\longrightarrow}0,$$ where $\alpha_{-2}=\overline{\partial}\alpha_{-1}-\frac{\langle \overline{\partial}\alpha_{-1},\alpha_{-1}\rangle}{|\alpha_{-1}|^2}\alpha_{-1}$.

By a similar discussion  we also claim
$$A^{''}_{\phi^{(2)}}(\underline{ker}A^{'\bot}_{\phi^{(2)}})\neq0.$$
Otherwise if $A^{''}_{\phi^{(2)}}(\underline{ker}A^{'\bot}_{\phi^{(2)}})=0$, then  $\phi^{(2)}$ is a Frenet pair from Lemma 3.3, and $\underline{\phi}^{(2)}=\underline{f}_0\oplus \underline{f}_1$. Since $\underline{\alpha}_{-1}\neq \underline{f}_1$, it is possible for us to put $$\alpha_{-1}=f_0+x_1f_1, \ \gamma=-\overline{x}_1|f_1|^2f_0+|f_0|^2f_1, \ \beta=f_2+x_2\gamma,$$
where $x_1$ and $x_2$ are  smooth functions on $S^2$ expect some isolated points.
So it verifies   $x_2=0$ from the relation $\langle \beta, f_{1}\rangle=0$, i.e. $\underline{\beta}=\underline{f}_{2}$, which  is a contradiction.
Thus $A^{''}_{\phi^{(2)}}(\underline{ker}A^{'\bot}_{\phi^{(2)}})\neq0$ holds.

Reusing the above methods,  $\underline{f}_{1}$ is a subbundle with rank one of $\underline{\phi}^{(2)}$, let $\underline{\gamma}^1 = \underline{f}^{\bot}_{1} \cap \underline{\phi}^{(2)}$, then $\underline{\phi}^{(2)}$ can be rewritten as
$\underline{\phi}^{(2)}=\underline{\gamma}^1\oplus\underline{f}_1$.
The backward replacement of $\underline{f}_{1}$ produces a new harmonic map $$\underline{\phi}^{(3)}=\underline{\alpha}_{-2}\oplus \underline{\gamma}^1: S^2\rightarrow G(2,N;\mathbb{C}),$$ it belongs to the following  harmonic sequence:
\begin{equation}0\stackrel{\partial{''}}
{\longleftarrow}\cdots\stackrel{\partial{''}}
{\longleftarrow}\underline{\phi}^{(3)}=\underline{\alpha}_{-2}\oplus \underline{\gamma}^1\stackrel{\partial{'}} {\longrightarrow}
\underline{f}_{1}\stackrel{\partial{'}} {\longrightarrow}
\underline{f}_{2}\stackrel{\partial{'}} {\longrightarrow}
\underline{f}_{3}\stackrel{\partial{'}} {\longrightarrow}
\underline{f}_{4}\stackrel{\partial{'}}
{\longrightarrow}0.\label{eq:5.25}\end{equation}
In  harmonic sequence \eqref{eq:5.25} we notice that $\underline{f}_0$ is a subbundle with rank one of $\underline{\phi}^{(3)}$, it satisfies $A^{''}_{\phi^{(3)}}|f_0=0$, i.e.
 $$A^{''}_{\phi^{(3)}}(\underline{ker}A^{'\bot}_{\phi^{(3)}})=0,$$
hence $\phi^{(3)}$ is a mixed pair.
Suppose therefore that $\underline{\phi}^{(3)}=\underline{g}_m\oplus \underline{f}_0$,
$\underline{g}_m$ and $\underline{f}_0$ are anti-holomorphic and holomorphic curves in $\mathbb{C}P^m$ and $\mathbb{C}P^n$ respectively such that $\underline{h}_1\bot \underline{g}_m$.
Then it is reasonable to put
$$\gamma^{1}=g_m+x_3f_0, \ \alpha_{-2}=-\overline{x}_3|f_0|^2g_m+|g_m|^2f_0, \ \gamma=\gamma^1+x_4f_1,$$ where $x_3$ and $x_4$ are smooth functions on $S^2$ expect some isolated points.
We thus derive $x_4=0$ from the fact $\langle \beta, f_{1}\rangle=0$, which gives
$ \gamma=\gamma^1,\ \underline{\alpha}_{-1}=\underline{f}_1, \ \underline{\alpha}_{-2}=\underline{f}_0$,
it further implies   $$\underline{\phi}^{(3)}=\underline{c}_0\oplus \underline{f}_0, \quad \underline{\phi}^{(1)}=\underline{c}_0\oplus \underline{f}_2,   \quad \underline{\phi}=\underline{f}_3\oplus \underline{\alpha},$$ where $c_0=(0,\ldots,0, 1)^T$.

Next set \begin{equation}\alpha=f_2+x_5c_0, \quad \beta=-\overline{x}_5f_2+|f_2|^2c_0,\label{eq:5.26}\end{equation} where $x_5$  is a smooth functions on $S^2$ expect some isolated points,
then using  \eqref{eq:5.26} we arrive at   
$$ \partial{x}_5=0,  \quad \overline{\partial}x_5=0, \quad |x_5|^2=2|f_2|^2=48,   \quad   \langle {\partial} \alpha, \alpha\rangle=0,   \quad    \langle \overline{\partial} \alpha, \alpha\rangle=0$$
according to 
$$\alpha_{-1}=\overline{\partial} \alpha-\frac{\langle \overline{\partial} \alpha, \alpha\rangle}{|\alpha|^2}\alpha,    \ \alpha_1 =\partial \alpha-\frac{\langle \partial \alpha, \alpha\rangle}{|\alpha|^2}\alpha =\frac{\langle \alpha, f_{2}\rangle}{|f_{2}|^2}f_3,   \ \frac{|\alpha_{-1}|^2}{|\alpha|^2}=\frac{|f_4|^2}{2|f_3|^2},  \ |f_2|^2=|V_2^{(4)}|^2=24$$
respectively, the last equation follows from \eqref{eq:2.10}.

This together with \eqref{eq:3.2} implies finally that such $\underline{\phi}=\underline{f}_3\oplus \underline{\alpha}$ is of second fundamental form, it belongs to case (3)
 in Theorem 1.3, which completes the proof.
\end{proof}

Summing Lemmas 5.2-5.4, we obtain Theorem 1.3 in Section 1.

\section{$\partial^{'}$-irreducible and $\partial^{''}$-irreducible minimal two-spheres with parallel second fundamental form}
Let $\phi$ be a conformal minimal immersion from $S^2$ to $G(2,N;\mathbb{C})$ with rank $\partial^{'}\underline{\phi} =$ rank $\partial^{''}\underline{\phi} =2$. Suppose $\phi$ is of parallel second fundamental form,  in this section we analyze $\phi$ by $r=1$ and $r\geq 2$ respectively, where $r$ is the isotropy order of $\phi$.

\textbf{(a) $r=1$}. Here we consider harmonic map $\phi: S^2\rightarrow G(2,N;\mathbb{C})$ of finite isotropy order $r=1$ under the supposition rank $\partial^{'}\underline{\phi} =$ rank $\partial^{''}\underline{\phi} =2$, then $\phi$ belongs to the following harmonic sequence
\begin{equation}0\stackrel{\partial{''}}
{\longleftarrow}\cdots\stackrel{\partial{''}}
{\longleftarrow}\underline{\phi}_{-1}\stackrel{\partial{''}}
{\longleftarrow}\underline{\phi}\stackrel{\partial{'}} {\longrightarrow}
\underline{\phi}_{1}\stackrel{\partial{'}}
{\longrightarrow}\cdots\stackrel{\partial{'}}
{\longrightarrow}0.\label{eq:6.1}\end{equation}
Since $r=1$, it is easy to see that there exists a local unitary frame $e_0, \ e_1, \ e_2, \ e_3, \ e_4$ of $S^2 \times \mathbb{C}^N$ such that $\{e_0, \ e_1\}$, $\{e_0, \ e_2\}$,  $\{e_3, \ e_4\}$ locally span subbundles $\underline{\phi}_1, \ \underline{\phi}_{-1}, \ \underline{\phi}$ of $S^2 \times \mathbb{C}^N$ respectively, and
$$\underline{\phi}_{-1}\cap \underline{\phi}_1=\underline{e}_0, \ A^{'}_{\phi}(\underline{e}_3)=\underline{e}_0.$$ Let  $W_{-1}=(e_0, \ e_2), \ W_{0}=(e_3, \ e_4), \ W_{1}=(e_0, \ e_1)$,  then
\begin{equation}W_1^{*}W_{-1} = \left(\begin{array}{ccccccc}
1 & 0 \\
0  & 0 \end{array}\right).\label{eq:6.2}\end{equation}
Observing the fact $\phi=W_0W_0^{*}$,  then by using \eqref{eq:2.4}, it is not difficult to get the following  equations:
$$ A_{\overline{z}}=W_{-1}\Omega_{-1}^{*}W_{0}^{*}+W_0\Omega_{0}^{*}W_{1}^{*},$$
$$
\begin{array}{lll} [A_{\overline{z}}, [A_z, A_{\overline{z}}]] & = &
-2W_{-1}\Omega_{-1}^{*}\Omega_{-1}\Omega_{-1}^{*}W_0^{*}-2W_{0}\Omega_{0}^{*}\Omega_{0}\Omega_{0}^{*}W_1^{*}+W_{-1}\Omega_{-1}^{*}\Omega_{0}^{*}\Omega_{0}W_0^{*}
\\
& & +W_{0}\Omega_{0}^{*}W_{1}^{*}W_{-1}\Omega_{-1}^{*}\Omega_{-1}W_{-1}^{*}
+W_{1}\Omega_{0}\Omega_{0}^{*}W_{1}^{*}W_{-1}\Omega_{-1}^{*}W_{0}^{*}
\\
& & +W_{0}\Omega_{-1}\Omega_{-1}^{*}\Omega_{0}^{*}W_1^{*}.\end{array} $$
Hence we immediately get
\begin{equation}  M_1\Omega_{-1}^{*}=-2\Omega_{-1}^{*}\Omega_{-1}\Omega_{-1}^{*}+\Omega_{-1}^{*}\Omega_{0}^{*}\Omega_{0}
+W_{-1}^{*}W_{1}\Omega_{0}\Omega_{0}^{*}W_{1}^{*}W_{-1}\Omega_{-1}^{*},\label{eq:6.3}\end{equation}
\begin{equation} M_1\Omega_{0}^{*}=-2\Omega_{0}^{*}\Omega_{0}\Omega_{0}^{*}+\Omega_{-1}\Omega_{-1}^{*}\Omega_{0}^{*}
+\Omega_{0}^{*}W_{1}^{*}W_{-1}\Omega_{-1}^{*}\Omega_{-1}W_{-1}^{*}W_{1} \label{eq:6.4}\end{equation}
from the fact $[A_{\overline{z}}, [A_z, A_{\overline{z}}]]=M_1A_{\overline{z}}$.
For convenience, using relation $A^{'}_{\phi}(\underline{e}_3)=\underline{e}_0$, it is possible for us to put
\begin{equation}\Omega_{-1} = \left(\begin{array}{ccccccc}
x & y \\
z  & w \end{array}\right),
\quad \Omega_{0} = \left(\begin{array}{ccccccc}
\lambda & \nu \\
0  & \mu \end{array}\right), \label{eq:6.5}\end{equation}
which are both $\left(2 \times 2\right)$-matrices of rank two. By comparing every elements of  matrix $\Omega_0$, condition  \eqref{eq:6.4} is equivalent to the equations
\begin{equation}\left\{
\begin{array}{l} M_1 \overline{\lambda} = \overline{\lambda}(-2|\lambda|^2-2|\nu|^2+2|x|^2+|y|^2+|z|^2)+\overline{\nu}(x\overline{z}+y\overline{w}), \\
0 = \overline{\mu}(x\overline{z}+y\overline{w}-2\overline{\lambda}\nu), \\
M_1\overline{\mu} = \overline{\mu}(-2|\nu|^2-2|\mu|^2+|z|^2+|w|^2),\\
M_1\overline{\nu} = \overline{\nu}(-2|\lambda|^2-2|\mu|^2-2|\nu|^2+|x|^2+2|z|^2+|w|^2)+\overline{\lambda}(\overline{x}z+\overline{y}w)
\end{array}\right. \label{eq:6.6}\end{equation}
 by  \eqref{eq:6.2} and  \eqref{eq:6.5}. From above equations in \eqref{eq:6.6}  we claim
\begin{equation}\nu=0.\label{eq:6.7}\end{equation}
Otherwise by using $\lambda\neq 0, \ \mu \neq 0$ and $\nu\neq0$ we derive
\begin{equation}\left\{
\begin{array}{l}
 x\overline{z}+y\overline{w}=2\overline{\lambda}\nu, \\
M_1 = -2|\mu|^2-2|\nu|^2+|z|^2+|w|^2,\\
M_1= -2|\mu|^2-2|\nu|^2+|x|^2+2|z|^2+|w|^2,
\end{array}\right. \label{eq:6.8}\end{equation}
 then $x=0, \ z=0$ can be obtained by comparing the second and the third equations of \eqref{eq:6.8}, which contradicts our supposition rank $\Omega_{-1}=2$ and then verifies \eqref{eq:6.7}.
So we have  the following conclusions
\begin{equation}\left\{
\begin{array}{l} M_1 = -2|\lambda|^2+2|x|^2+|y|^2+|z|^2, \\M_1 = -2|\mu|^2+|z|^2+|w|^2,\\
 x\overline{z}+y\overline{w}=0
\end{array}\right. \label{eq:6.9}\end{equation}
hold.

Next we   consider \eqref{eq:6.3}  by induction. Applying \eqref{eq:6.2}  and \eqref{eq:6.5}, condition \eqref{eq:6.3}   holds if and only if the following equations
\begin{equation}(M_1+2|x|^2+2|y|^2-2|\lambda|^2)\overline{x} = 0,\label{eq:6.10}\end{equation}
\begin{equation}(M_1+2|x|^2+2|y|^2-|\lambda|^2)\overline{y} = 0,\label{eq:6.11}\end{equation}
\begin{equation}(M_1+2|z|^2+2|w|^2-|\lambda|^2-|\mu|^2)\overline{z} = 0,\label{eq:6.12}\end{equation}
\begin{equation}(M_1+2|z|^2+2|w|^2-|\mu|^2)\overline{w} = 0\label{eq:6.13}\end{equation}
hold, which implies $x=0$ or $y=0$
by observing \eqref{eq:6.10}  and \eqref{eq:6.11}.

Here we claim \begin{equation} x=0.\label{eq:6.14}\end{equation} In order to achieve our objective, we first discuss the case $x\neq 0$. In view of \eqref{eq:6.10}-\eqref{eq:6.13} we get $y=z=0$ and $w\neq 0$.
\eqref{eq:6.10} becomes
$M_1=2|\lambda|^2-2|x|^2$, comparing it with the first equation of  \eqref{eq:6.9},  $M_1=0$, which is impossible. At present, we have made \eqref{eq:6.14} true,
then it can be clearly seen that \eqref{eq:6.3} becomes
\begin{equation}\left\{
\begin{array}{l}
 w=0, \\
M_1 = |\lambda|^2-2|y|^2,\\
M_1= |\lambda|^2+|\mu|^2-2|z|^2.
\end{array}\right. \label{eq:6.15}\end{equation}
Furthermore by  \eqref{eq:6.9} and  \eqref{eq:6.15}, we get
$|y|^2=|\mu|^2=-2M_1, \quad |z|^2=|\lambda|^2=-3M_1$,  which implies
$$L_{-1}=L_0=-5M_1$$ by using $L_{-1} =\textrm{tr} (\Omega_{-1} \Omega^*_{-1})$ and $L_{0} =\textrm{tr} (\Omega_{0} \Omega^*_{0})$.
It shows that $\phi: S^2 \rightarrow G(2, N; \mathbb{C})$ is totally real with $\nabla B=0$ and $r=1$. Then from \cite{Berndt}  and Theorem 1.1 of \cite{JGA},
$$\underline{\phi}=\underline{UV}^{(4)}_1\oplus \underline{UV}^{(4)}_3: S^2 \rightarrow G(2,5; \mathbb{C})$$  for some $U\in U(5)$ satisfies $\underline{\overline{UV}}^{(4)}_0=\underline{UV}^{(4)}_4$,
 which is totally geodesic with $K=\frac{1}{5}$, and it is congruent to the case (1) in Theorem 1.4.

\textbf{(b) $r\geq 2$}. In this part we consider $\partial^{'}$-irreducible and $\partial^{''}$-irreducible harmonic map $\phi: S^2\rightarrow G(2,N;\mathbb{C})$ of isotropy order $r\geq 2$ (including the strongly isotropic case), here  $\phi$ also derives the harmonic sequence given in \eqref{eq:6.1}. Since $r\geq 2$, $\phi_{-1}, \ \phi$ and $\phi_1$ are mutually orthogonal and
$W_1^{*}W_{-1} = 0$.
Similar calculations give
$$
\begin{array}{lll} [A_{\overline{z}}, [A_z, A_{\overline{z}}]] & = &
-2W_{-1}\Omega_{-1}^{*}\Omega_{-1}\Omega_{-1}^{*}W_0^{*}-2W_{0}\Omega_{0}^{*}\Omega_{0}\Omega_{0}^{*}W_1^{*}
\\
& & +W_{-1}\Omega_{-1}^{*}\Omega_{0}^{*}\Omega_{0}W_0^{*} +W_{0}\Omega_{-1}\Omega_{-1}^{*}\Omega_{0}^{*}W_1^{*}.\end{array} $$
At this time, under the assumption that $\phi$ is of parallel second fundamental form, it following from $[A_{\overline{z}}, [A_z, A_{\overline{z}}]]=M_1A_{\overline{z}}$ that
$$\left\{
\begin{array}{l} M_1I = -2\Omega_{-1}\Omega_{-1}^{*}+\Omega_{0}^{*}\Omega_{0}, \\
M_1I = -2\Omega_{0}^{*}\Omega_{0}+\Omega_{-1}\Omega_{-1}^{*},
\end{array}\right. $$
where $I$ is a $\left(2 \times 2\right)$-identity matrix, which implies
$L_{-1}=L_{0}=-2M_1$.
It concludes that $\phi: S^2 \rightarrow G(2, N; \mathbb{C})$ is totally real with parallel second fundamental form and isotropy order $r\geq 2$. Then from \cite{Berndt, HPn1} and \cite{JGA} , $\phi$ belongs to the following two cases:
\\ (1) $\underline{\phi}=\overline{\underline{UV}}^{(2)}_1\oplus \underline{UV}^{(2)}_1: S^2 \rightarrow G(2,6;\mathbb{C})$ or  $\underline{\phi}=\textbf{\underline{J}}\underline{UV}^{(2)}_1\oplus \underline{UV}^{(2)}_1: S^2 \rightarrow G(2,6;\mathbb{C})$ for some $U\in U(6)$;
\\ (2) $\underline{\phi}=\overline{\underline{UV}}^{(4)}_2\oplus \underline{UV}^{(4)}_2: S^2 \rightarrow G(2,10;\mathbb{C})$ or  $\underline{\phi}=\textbf{\underline{J}}\underline{UV}^{(4)}_2\oplus \underline{UV}^{(4)}_2: S^2 \rightarrow G(2,10;\mathbb{C})$  for some $U\in U(10)$.

In the absence of confusion, adding zeros to $V^{(2)}_1$, let $\widehat{V}^{(2)}_1=\frac{1}{1+z\overline{z}}(0, \ 0, \ 0, \ -2\overline{z}, \ \sqrt{2}(1-z\overline{z}), \ 2z)^T$ and $V^{(2)}_1=\frac{1}{1+z\overline{z}}(-2\overline{z}, \ \sqrt{2}(1-z\overline{z}), \ 2z, \ 0, \ 0, \ 0)^T$, in summary, for some $U\in U(6)$, $\phi$ shown in (1) can be expressed as
$$\underline{\phi}=\underline{U}\underline{\widehat{V}}_1^{(2)}\oplus \underline{UV}_1^{(2)}: S^2 \rightarrow G(2,6; \mathbb{C}),$$ which is totally geodesic with $K=\frac{1}{2}$ by a
series of calculations. It is congruent to the case (2) in Theorem 1.4.

Similarly, let $$\widehat{V}^{(4)}_2=\frac{2}{(1+z\overline{z})^2}(0, \ 0, \ 0, \ 0, \ 0, \ 6\overline{z}^2, \ 6\overline{z}(z\overline{z}-1), \ \sqrt{6}(1-4z\overline{z}+(z\overline{z})^2), \ 6z(1-z\overline{z}), \ 6z^2)^T,$$ 
 $$V^{(4)}_2=\frac{2}{(1+z\overline{z})^2}(6\overline{z}^2, \ 6\overline{z}(z\overline{z}-1), \ \sqrt{6}(1-4z\overline{z}+(z\overline{z})^2), \ 6z(1-z\overline{z}), \ 6z^2, \ 0, \ 0, \ 0, \ 0, \ 0)^T,$$
in summary, for some $U\in U(10)$,  $\phi$ in (2) can be expressed as
$$\underline{\phi}=\underline{U}\underline{\widehat{V}}_2^{(4)}\oplus \underline{UV}_2^{(4)}: S^2 \rightarrow G(2,10; \mathbb{C}),$$ which is of parallel second fundamental form  with $K=\frac{1}{6}$ and $\|B\|^2=\frac{2}{3}$, and it is congruent to the case (3) in Theorem 1.4.

Summing up, we get Theorem 1.4 in Section 1.

Theorems 1.1-1.4 in Section 1 determine all conformal minimal immersions of parallel second fundamental form from $S^2$ to $G(2,N;\mathbb{C})$.
It is easy to check that no two of these eighteen cases are congruent, i.e., we can not transform any one into another by  left multiplication by U(N).

Up to an isometry of $G(2,N;\mathbb{C})$,  Theorems 1.1-1.4 show that all   linearly full conformal
minimal immersions of parallel second fundamental form from $S^2$ to
$G(2,N;\mathbb{C})$ are presented by  Veronese surfaces
in $\mathbb{C}P^n$ for some $n < N$. It is easy to check that these eighteen
minimal immersions are all homogeneous. Of course they
contain those given by (\cite{HPn1}, Theorem 1.1 and \cite{JGA}, Theorem 1.1), even more than those (cf. cases (1)(2) and (3) shown in Theorem 1.1 etc.).


\begin{thebibliography}{90}
\bibitem{Berndt}
J. Berndt, \emph{Riemannian geometry of complex two-plane Grassmannians}, Rend. Sem. Mat. Univ. Pol. Torino,
55(1)(1997), 19-83.
\bibitem{Bolton}
J. Bolton, G.R. Jensen, M. Rigoli and L.M. Woodward, \emph{On
conformal minimal immersions of $S^2$ into $CP^n$}, Math. Ann.,
279(1988), 599-620.
\bibitem{Burstall}
F.E. Burstall and J.C. Wood, \emph{The construction of  harmonic
maps  into  complex Grassmannians}, J. Diff. Geom., 23(1986),
255-297.
\bibitem{Chern}
S.S. Chern  and J. G. Wolfson, \emph{Minimal surfaces by moving
frames}, Amer. J. Math., 105(1983), 59-83.
\bibitem{Chi}
Q.S. Chi and Y.B. Zheng, \emph{Rigidity of pseudo-holomorphic curves
of constant curvature in Grassmann manifolds}, Trans. Amer. Math.
Soc., 313(1989), 393-406.

\bibitem{D-H-Z1}
 L. Delisle, V. Hussin and W.J. Zakrzewski, \emph{Constant curvature solutions of Grassmannian sigma models: (1) Holomorphic solutions}, J. Geom. Phys., 66(2013), 24-36.
\bibitem{D-H-Z2} L. Delisle, V. Hussin and W.J. Zakrzewski, \emph{Constant curvature solutions of Grassmannian sigma models: (2) Non-holomorphic solutions}, J. Geom. Phys., 71(2013), 1-10.

\bibitem{Erdem}
S. Erdem and J.C. Wood, \emph{On the construction of harmonic maps
into a Grassmannian}, J. London Math. Soc., 28(1)(1983), 161-174.
\bibitem{Fei}
J. Fei, X.X. Jiao, L. Xiao  and X.W. Xu, \emph{On the classification of homogeneous 2-spheres in complex Grassmannians}, Osaka J. Math.,
50(2013), 135-152.
\bibitem{Ferus}
D. Ferus, \emph{Immersions with parallel second fundamental form}, Math. Z.,
140(1974), 87-93.
\bibitem{Ferus1}
D. Ferus, \emph{Symmetric submanifolds of Euclidean space}, Math. Ann.,
247(1)(1980),  81–93.
\bibitem{HP2}
L. He and X.X. Jiao, \emph{Classification of conformal minimal
immersions of constant curvature from $S^2$ to $HP^2$}, Math. Ann.,
359(3-4)(2014), 663-694.
\bibitem{HPn1}
L. He and X.X. Jiao, \emph{On conformal minimal
immersions of  $S^2$ in $HP^n$ with parallel second fundamental
form}, Ann.Mat. Pur. Appl., 194(5)(2015), 1301-1317.
\bibitem{Jiao2008}
X.X. Jiao, \emph{Pseudo-holomorphic curves of constant curvature in
complex Grassmannians}, Israel J. Math., 163(1)(2008), 45-63.
\bibitem{Jiao2012}
X.X. Jiao, \emph{On minimal two-spheres immersed in complex Grassmann manifolds with parallel second
fundamental form}, Monatsh. Math., 168(2012), 381-401.
\bibitem{JGA}
X.X. Jiao and M.Y. Li, \emph{On conformal minimal immersions of two-spheres in a
complex hyperquadric with parallel second fundamental form}, J. Geom. Anal., 26(2016):185-205.
\bibitem{Q4}
X.X. Jiao, M.Y. Li and H. Li, \emph{Rigidity of conformal minimal immersions of
constant curvature from $S_2$ to $Q_4$}, J. Geom. Anal., 31(2021):2212-2237.
\bibitem{Kagan}
V.F. Kagan, \emph{The fundamentals of the theory of surfaces in tensor presentation. Part two. Surfaces in space. Transformations and Deformations of surfaces. Special questions (Russian)}, Gosudarstv. Izdat. Tehn.-Teor. Lit., Moscow-Leningrad,
1948.
\bibitem{Naitoh}
H. Naitoh, \emph{Parallel submanifolds of complex space forms. II}, Nagoya Math. J., 91 (1983), 119–149. 
\bibitem{Nakagawa}
H. Nakagawa and R. Takagi, \emph{On locally symmetric Kaehler
submanifolds in a complex projective space}, J. Math. Soc. Japan,
28(4)(1976), 638-667.
\bibitem{Ros}
A. Ros, \emph{On spectral geometry of Kaehler submanifolds}, J. Math. Soc. Japan,
36(3)(1984), 433-447.
\bibitem{Takeuchi}
M. Takeuchi,  \emph{Parallel submanifolds of space forms}, Manifolds and Lie groups (Notre Dame, Ind., 1980), Progr. Math., 14, Birkhäuser, Boston, Mass., 1981, 429-447.
\bibitem{Uhlenbeck}
K. Uhlenbeck, \emph{Harmonic maps into Lie groups (classical
solutions of the chiral model)}, J. Diff. Geom., 30(1989), 1-50.
\bibitem{Wolfson}
J. G. Wolfson, \emph{Harmonic sequences and harmonic maps of surfaces into complex Grassmann manifolds}, J. Diff. Geom., 27(1988), 161-178.
\end{thebibliography}
\end{document}